\documentclass [11 pt]{amsart}
\usepackage{graphicx} 
\usepackage{amsmath,amsfonts}
\usepackage{amssymb,amsbsy,mathrsfs,amscd}
\usepackage{color}
\usepackage{geometry}
\usepackage{mathtools}
\usepackage{amsfonts,bm}
\usepackage[all,cmtip]{xy}
\usepackage{pgfplots}
\usepackage{caption}
\usepackage{framed}
\newenvironment{mathframed}{\framed%
\allowdisplaybreaks
\vspace*{-\abovedisplayskip}\noindent}{%
\vspace*{-\dimexpr\baselineskip+\topsep}\endframed}

\usepackage{tikz-cd}
\usetikzlibrary{matrix}

\allowdisplaybreaks[4]

\newcommand{\bu}{\bm u}

\newcommand{\bv}{\bm v}

\newcommand{\bbR}{\mathbb{R}}
\newcommand{\nab}{\nabla}

\newcommand{\Del}{\Delta}

\newcommand{\p}{\partial}

\newcommand{\bw}{\bm w}

\newcommand{\mct}{\mathcal{T}_N}

\newcommand{\bigzero}{\mbox{\normalfont\Large\bfseries 0}}
\newcommand\norm[1]{\left\lVert#1\right\rVert}
\numberwithin{equation}{section}

\newtheorem{theorem}{Theorem}[section]

\theoremstyle{definition}

\theoremstyle{remark}
\newtheorem{remark}[theorem]{Remark}

\newcommand{\bbN}{\mathbb{N}}

\begin{document}
\title[Null controllability FEM vs FDM]{Numerical Approximations for the Null Controllers of Structurally Damped Plate Dynamics}
\author[P. G. Geredeli, C. Givens and A. Zytoon]{Pelin G. Geredeli\email{\lowercase{peling@iastate.edu}}, Carson Givens and Ahmed Zytoon\email{\lowercase{zytoon@iastate.edu}}}
\address{Department of Mathematics, Iowa State University, Ames, IA 50011}

\maketitle

\begin{abstract}
In this paper, we consider a structurally damped elastic equation under hinged boundary conditions. Fully-discrete numerical approximation schemes are generated for the null controllability of these parabolic-like PDEs. We mainly use finite element method (FEM) and finite difference method (FDM) approximations to show that the null controllers being approximated via FEM and FDM exhibit exactly the same asymptotics of the associated minimal energy function. For this, we appeal to the theory originally given by R. Triggiani \cite{1} for construction of null controllers of ODE systems. These null controllers are also amenable to our numerical implementation in which we discuss the aspects of FEM and FDM numerical approximations and compare both methodologies. We justify our theoretical results with the numerical experiments given for both approximation schemes. 
\end{abstract}


\thispagestyle{empty}
\section{\textbf{Introduction}}

The partial differential equations (PDEs) of plate dynamics ubiquitously arise in elasticity to model and describe the oscillations of thin structures with large transverse displacements \cite{Lagnese}. Moreover, researchers of PDE control theory are often interested in devising control input methodologies by which one can elicit some pre-assigned behavior with respect to solutions of a given controlled plate or bean PDE system. In the course of constructing such a control theory for the given damped or undamped plate PDE, its underlying characteristics -hyperbolic or parabolic- must necessarily be taken into account \cite{lt}.

For example, whereas in hyperbolic equations, we have the notion of finite speed of propagation and evolution of singularities, the parabolic equations posses infinite speed of propagation and smoothing effect. In consequence, the notion of exact controllability-i.e., steering initial data to any finite energy state at some time (large enough) - is a reasonable object of study for hyperbolic problems. On the other hand, the null controllability problem- steering the initial data to the zero state at any time- makes sense for parabolic problems due to their smoothing effects.

In particular, there has been a great interest in studying the null controllability of infinite dimensional systems \cite{NCGAP, 13, 14, glo, 12, 1} with a view towards attaining optimal estimates for norms of minimal norm steering controls. In particular, null-controllability for deterministic parabolic-like PDE dynamics plays a crucial role in connection with corresponding stochastic parabolic differential equations. For example, it is known that the notion of null-controllability is equivalent to the strong Feller property of the semigroup of transition of the corresponding stochastic differential equation, which is obtained from the deterministic one by simply replacing the deterministic control with stochastic noise \cite{ek1, ek2, ek3}.

This manuscript considers certain PDE dynamics which exhibit analytic, or ‘‘parabolic-like’’ features. Since these dynamics are associated with an infinite speed of propagation (see \cite{LT}), it seems natural to ask: ``Is there any control function which steers the solution to the zero state after some certain time $T>0?$" This is the problem of “null controllability”. However, we must distinguish the ``null controllability'' concept between finite and infinite dimensional (PDE) systems since while the issue of finding asymptotics  for the associated minimal energy function defined in \eqref{minnorm} has completely been characterized in the finite dimensional ODE case \cite{24, 26}, the infinite dimensional PDE case is in general an open problem. \cite{24} provides a formula which describes the growth of the minimal norm control, as time $T\rightarrow 0$ for ODE dynamics. This result depends on the Kalman’s rank condition, which is the sufficient and necessary controllability condition in finite dimensions. In the case of interior boundary control, it was proved in \cite{1} there is a relation between the infinite dimensional asymptotics and finite dimensional truncations such that a priori bounds manifested by the approximating sequence of null controllers (for finite dimensional system) will lead to the conclusion of a null controller for the (infinite dimensional) analytic PDE systems under consideration. It was also shown in \cite{1} that infinite dimensional null controllers will capture the sharp asymptotics of the associated minimal energy function, which is defined through the means of minimal norm controls (see \eqref{minnorm}).

The numerical approximation of controlled PDEs has been a topic of longstanding interest \cite{glo} however in contrast to the growing literature on theoretical results obtained for the null controllability of parabolic-like plate equations, the knowledge about numerical approximation of the null controllability of PDE dynamics which exhibit analytic, or ‘‘parabolic-like’’ features is relatively limited. In \cite{NCGAP} semidiscrete finite element method (FEM) approximation scheme were presented for the null controllability of  non-standard parabolic PDE systems. The key feature in \cite{NCGAP} is that the approximating null controllers exhibit the asymptotics of the associated minimal energy function for the fully infinite dimensional system.

In this manuscript, our main goals are to derive fully-discrete Finite Element Method (FEM) and Finite Difference Method (FDM) numerical approximation schemes for a certain (nonstandard) analytic and parabolic-like PDE system, give numerical implementation, and compare the respective FEM and FDM approximations for this controlled structurally damped elastic equation. The main novelties of the current work are:\\

\textbf{(i) Fully discrete FEM Approximation:} The PDE model given in \eqref{MainNC} below was firstly studied in \cite{NCGAP}. It was proved that certain finite element method
(FEM) approximations $\{u_N^{*}\}$ and their limiting controller $\{u^{*}\}$ for the structurally damped PDE \eqref{MainNC} manifest the asymptotics (given in Theorem \ref{thm:Blowuprate}) of $\mathcal{E}_{min}(T)$ defined in \eqref{minnorm}. However, in this work no numerical implementation was provided for the derived FEM scheme. In the present work, unlike the semi-discrete approximations, we use ``fully-discrete"  FEM approximation and provide a numerical experiment to justify that the approximation of the null controllers, within FEM numerical scheme framework, obey the same blow up rate of $\mathcal{O}(T^{-3/2})$ given in Theorem \ref{thm:Blowuprate}. Moreover, we give an explicit formula for the approximate control functions. \\

\textbf{(ii) Fully discrete FDM Approximation:} We numerically analyze the null controllability problem for the given PDE \eqref{MainNC} below by means of the finite difference method approximation scheme. We see that Theorem \ref{thm:Discrete} can be employed to justify the use of finite difference method (FDM) approximations to numerically recover a solution to the said null controllability problem. In particular, we provide a theoretical proof for our main result Theorem \ref{thm:FDM} which essentially states that the approximating null controllers are uniformly bounded ``in $N$'' by the minimal energy assymptotics for the fully infinite dimensional controlled PDE system \eqref{MainNC}. Subsequently, using fully discrete FDM approximation scheme, we construct explicit control functions and give the numerical implementation.\\

\textbf{(iii) Comparison of FEM vs. FDM:} Since the numerical approximation of controlled PDEs is a topic of longstanding interest, a natural question arises: which numerical approximation method would give a better result to see that the infinite dimensional control $u^*,$ a control which inherits the finite dimensional asymptotics? Our numerical implementations for FEM and FDM approximations yield that while the finite difference method scheme (FDM) gives better results in approximating the control function at terminal time $T$, the finite element method scheme (FEM) is more stable in computing the control across different values of $T$.
\vspace{0.5cm}

\textbf{Plan of the Paper.} In Section 1, we introduce the PDE model under consideration and describe the mathematical setting to be used throughout the manuscript. We also recall the key theory given in \cite{1} to which we will appeal in proving our results. Since one of our main results is the numerical implementation of the finite element method approximation scheme, we will refer to the semi-discrete variational formulations generated within this framework in \cite{NCGAP}. We provide the reader the entire FEM scheme in Section 2. Then in Section 3, we consider the application of Theorem \ref{thm:Discrete} within the Finite Difference approximation scheme. For this, we prove Theorem \ref{thm:FDM} which guarantees the existence of null controllers for the finite difference method (FDM) approximating system. Section 4 is devoted to the numerical implementation of the both finite element method (FEM) and finite difference method (FDM) approximation schemes. We also give the algorithmic description of those schemes. In Section 5 and 6, we give our numerical experiments and conclusions, respectively. We mainly compare the two FEM and FDM numerical approximation schemes to understand which method is more stable and gives better results in approximating the null controllers of corresponding systems. In the last Section, we give a very clean and easy to follow recipe to construct a numerical test problem to the (homogeneous part) PDE \eqref{MainNC} below. For this, we appeal to algebraic theory to compute the matrix exponential that represents the solution to the PDE \eqref{MainNC}.
\vspace{0.5cm}

\noindent Throughout the paper the norms $||\cdot ||$ are taken to be $%
L^{2}(D)$ for the domain $D$, and the inner products in $L^{2}(D)$ is written $%
(\cdot ,\cdot )$. The space $H^{s}(D)$ will denote the Sobolev
space of order $s$, defined on a domain $D$, and $H_{0}^{s}(D)$ denotes the
closure of $C_{0}^{\infty }(D)$ in the $H^{s}(D)$ norm which we denote by $\Vert \cdot \Vert _{s,D}$. Also, $C$ will denote a generic positive constant. For any $T>0$, we recall the space 
\begin{align*}
L^2(0, T; L^2(\Omega)):=\{w:\Omega\times [0,T] \mapsto \bbR: w(.,t) \in L^2(\Omega), \forall t\in[0,T], \int_0^T\ {\|w(t)\|_{{\it L}^2(\Omega)}^2}\,dt<\infty \}.
\end{align*}

\noindent In what follows, $\Omega \subset \bbR^2$ will be a bounded polygonal domain with Lipschitz continuous boundary $\p\Omega = \Gamma$ and we consider the following controlled PDE system:
\begin{subequations}
\label{MainNC}
\begin{alignat}{2}
\label{MainNC1}
\omega_{tt} + \Delta^2 \omega -\rho\Delta\omega & = u,\quad &&\text{on }\Omega\times (0,T),\\
\label{MainNC2}
\omega = \Delta\omega & =0, \quad &&\text{on } \Gamma\times (0,T),\\
\label{MainNC3}
[\omega(0), \omega_t(0)] & = [\omega_0, \omega_1].
\end{alignat}
\end{subequations}
Here $\omega = \omega(x,t)$ is the elastic plate variable which satisfies the ``hinged boundary conditions", and the constant $ \rho>0~ (\rho \ne 2)$. The associated finite energy (Hilbert) space is given as
$$H=[H^2(\Omega) \cap H_0^1(\Omega)]\times L^2(\Omega).$$
We observe that the system \eqref{MainNC} can be rewritten as the ODE
\begin{equation}
\label{1stODE}
\frac{d}{dt}\begin{bmatrix}
\omega \\
\omega_t
\end{bmatrix}=\begin{bmatrix}
0 & I  \\
-A^2 & -\rho A 
\end{bmatrix}\begin{bmatrix}
\omega \\
\omega_t
\end{bmatrix}+\begin{bmatrix}
0 \\
u
\end{bmatrix},\quad 
\begin{bmatrix}
\omega(.,0) \\
\omega_t(.,0)
\end{bmatrix}\in H,
\end{equation}
where $A : D(A) \subset L^2(\Omega)\mapsto L^2$ is the (homogeneous) ``Dirichlet Laplacian''
\begin{equation}\label{DL}
Af = -\Delta f, \qquad D(A) = H^2(\Omega) \cap H_0^1(\Omega).
\end{equation}
Alternatively, the system \eqref{1stODE} will be equivalent, via the change of variables $$v = A\omega,~~ w = \omega_t,$$ to the following ODE:  
\begin{equation}
\label{2ndODE}
\frac{d}{dt}\begin{bmatrix}
v \\
w
\end{bmatrix}=\begin{bmatrix}
0 & A  \\
-A & -\rho A 
\end{bmatrix}\begin{bmatrix}
v \\
w
\end{bmatrix}+\begin{bmatrix}
0 \\
u
\end{bmatrix},\quad 
\begin{bmatrix}
v(0) \\
w(0)
\end{bmatrix} = \begin{bmatrix}
v_0 \\
w_0
\end{bmatrix}=\begin{bmatrix}
A\omega_0 \\
\omega_1
\end{bmatrix} \in L^2(\Omega) \times L^2(\Omega).
\end{equation}
An easy application of the Lumer-Phillips Theorem yields that there exists a unique solution $[v,w]\in L^2(\Omega) \times L^2(\Omega)$ to \eqref{2ndODE}, and subsequently $[A^{-1}v, w]=[\omega, \omega_t]$ in \eqref{1stODE} (or \eqref{MainNC}) have the regularity $ [\omega, \omega_t] \in C([0, T]; H).$ The dynamical system \eqref{2ndODE} was also shown to generate an analytic semigroup [9,10] which implies that the null controllability problem is the steering problem to be considered. In this regard, it was proved in \cite{1,12,13} that the following problem is solvable: \\
 
 \textbf{NC:} \textit{``Let terminal time $T>0$ be arbitrary. Given initial data $[\omega_0, \omega_1]\in H$, find $u\in L^2(0, T; L^2(\Omega))$ such that the corresponding solution $[\omega, \omega_t]$ of \eqref{MainNC} satisfies $$[\omega(T), \omega_t(T)] = [0,0].$$}
\noindent What is more, one can find the minimal norm control asymptotics relative to \eqref{2ndODE}. That is, --find $u_T^*(0, T; [\omega_0,\omega_1]) \in L^2(0, T; L^2(\Omega))$ such that $u_T^*$ solves the null controllability problem and minimizes the $L^2$-cost with respect to all possible null controllers-- Thus, the following ``minimal energy function'' is well defined:
\begin{equation}\label{minnorm}
\mathcal{E}_{min}(T) = \sup_{{\bf{x_0}}\in H, \| {\bf{x_0}} \|_H = 1} \|  u_T^*({\bf{x_0}}) \|_{L^2(0, T; L^2(\Omega))}.
\end{equation}
The reader is referred to the references \cite{1,14} for detailed information, however we will recall the following theorem that is related to the blow up rate of $\mathcal{E}_{min}(T)$.
\begin{theorem}\label{thm:Blowuprate}
(\cite{1,14}). The null controllability problem \textbf{(NC)} admits of a solution, and the associated minimal energy function $\mathcal{E}_{min}(T)$ given in \eqref{minnorm} obeys the blow up rate $\mathcal{O}(T^{-3/2})$. That is;
\begin{equation}\label{blup}
\mathcal{E}_{min}(T) = \sup_{{\bf{x_0}}\in H, \| {\bf{x_0}} \|_H = 1} \|  u_T^*({\bf{x_0}}) \|_{L^2(0, T; L^2(\Omega))}=\mathcal{O}(T^{-\frac{3}{2}}).
\end{equation}
\end{theorem}
The proof of Theorem \ref{thm:Blowuprate} was given in \cite{1,13,14} via using different techniques. While the weighted operator theoretic multiplier method and the analyticity of the corresponding semigroups --based on a continuous line of argument-- are utilized in \cite{13,14}, the proof in \cite{1} depends upon a discrete approach which was also used for the validation of the spectral truncations to the controlled dynamical system \eqref{MainNC}. Since our main goal here is to show ``numerically" that each finite dimensional approximating null controller (in the FEM and FDM approximation scheme) and their limiting controller manifest the same asymptotics as the minimal energy function $\mathcal{E}_{min}(T)$ for the infinite dimensional system \eqref{MainNC}, for the sake of clarity, we will provide below the following detailed theory that we will utilize:\\

\noindent Consider the following finite dimensional control system:
\begin{equation}\label{DiscRes1}
Y_N^{'} = \mathcal{A}_{N} Y_N + \mathcal{B}_{N} U_N, \qquad Y_N(0) = Y_{N,0} \in \bbR^{(k+1)N}, \qquad N = 1,2,...,
\end{equation}
where $\mathcal{A}_{N}$ is $(k+1)N \times (k+1)N$ and $\mathcal{B}_{N}$ is $(k+1)N \times N$ matrices, and the control $U_N \in L^2(0, T; \bbR^{N\times 1})$. Also, define the following $(k+1)N \times (k+1)N$ \textit{Kalman} matrix $\mathcal{K}_N$ as
\begin{equation}\label{KMatrix}
\mathcal{K}_N = [\mathcal{B}_{N}, \mathcal{A}_{N}\mathcal{B}_{N}, \mathcal{A}_{N}^2\mathcal{B}_{N}, ... , \mathcal{A}_{N}^k \mathcal{B}_{N}].
\end{equation}
It was shown in \cite{4} that if $\mathcal{K}_{N}$ has full rank for any $N = 1,2,...$, then the system \eqref{DiscRes1} is exactly controllable by means of controls in $L^2(0,T; \bbR^{N})$. Also, the control function $u_N^*(t)$ which steers the initial data $Y_{N,0}$ to the origin in given time $T>0$ was constructed in \cite{5} as follows:\\
 
 \noindent Define the scalar-valued function $f_T(t)$ and the $(k+1)N$ vector $\mu(t)$ as  

\begin{equation}\label{scalarf}
f_T(t) = \frac{t^k(T-t)^k}{C_{T,k}}, \qquad C_{T,k} = \int_0^T \ {t^k(T-t)^k} \,dt,
\end{equation}
and
\begin{equation}\label{controlvec}
\mu_N(t) = \begin{bmatrix}
\mu_0(t) \\
\mu_1(t) \\
\mu_2(t) \\
\vdots      \\
\mu_k(t) 
\end{bmatrix}= -\mathcal{K}_N^{-1} e^{\mathcal{A}_{N} t} Y_{N,0} f_T(t), \qquad 0\le t \le T,
\end{equation}
where each component $\mu_j$ is an $N-$vector. It was proved in \cite{5} that the choice of the following type of control function in \eqref{DiscRes1} will indeed steers the initial data $Y_{N,0}$ to the origin. \begin{equation}\label{control1}
u_N^*(t) = \mu_0(t) + \mu_1^{'}(t) + \mu_2^{''}(t) + \cdots + \mu_k^{(k)}(t).
\end{equation}
That is, the solution $Y_N$ of \eqref{DiscRes1} with control $u_N^*(t)$ satisfies the terminal condition $Y_N(T)=0.$ With this type of control function $u_N^*(t)$ in mind, we recall the following result from \cite{1} which will be the main ingredient in the application of our numerical schemes: 
\begin{theorem}\label{thm:Discrete}
With reference to the system \eqref{DiscRes1}, assume that the following conditions hold:\\

\noindent \textbf{(A1)} The Kalman matrix $\mathcal{K}_N$ satisfies the Kalman rank condition with index $k$. That is, 
 $$Rank(K_N)=(k+1)N~ for~ N = 1,2,...$$
 \textbf{(A2)} There exists $C_k>0$ independent of $N$ such that
\begin{equation}\label{invbound}
\| \mathcal{K}_N^{-1} \| \le C_k,
\end{equation}
\textbf{(A3)} There exists a constant $E_k$ such that 
\begin{equation}\label{ANbound}
\| \mathcal{A}_N^{j} e^{\mathcal{A}_{N} t} \| \le \frac{E_k}{t^j},\qquad (uniformly ~~ in~~ N) \qquad j = 0,1,...,k.
\end{equation}
Then the steering controls provided in \eqref{control1} obey the estimate
\begin{equation}\label{controlbound}
\Big( \int_0^T \ {\| u_N^*(t) \|^2} \,dt \Big)^{\frac{1}{2}} \le C_k^* \frac{\| Y_{N,0} \| }{T^{k + \frac{1}{2}}},
\end{equation}
where $C_k^*$ is a positive constant independent of $N = 1,2,...$.
\end{theorem}

\section{\textbf{Preliminaries}}

As mentioned in Section 1, it was theoretically (without any numerical experiment) shown in \cite{NCGAP} that Theorem \ref{thm:Discrete} can be employed to justify the use of finite element method approximations to numerically recover a solution to the null controllability problem \textbf{(NC)}. Our main goal here is to compare two numerical approximation schemes FEM vs FDM to see that the approximations $\{u_N^*\}$ and their limiting controller $u^*$ manifest the same asymptotics of the minimal energy function $\mathcal{E}_{min}(T)$ given in \eqref{minnorm}. For this, we will apply the FEM and FDM methodologies to the finite dimensional control system \eqref{DiscRes1} separately. The theoretical justification of the use of FEM approximation was already given in \cite[Theorem 4]{NCGAP}. Since we will refer to this scheme in the FEM numerical implementation, for the completeness and the convenience of the readers, we will remind it here: 
\subsection{Finite Element Method (FEM) Approximation Scheme for \eqref {DiscRes1}:\\ \\ Application of Theorem \ref{thm:Discrete}  }
\vspace{0.3cm}

Let $\mct$ be a triangulation (mesh) of $\Omega$, where $N$ is the number of vertices (nodes) in the triangulation $\mct$. For a triangle (element) $K\in \mct$,
we denote by $h_K = {\rm diam}(K)$ and set $h = \max_{K\in \mct} h_K$. We make the classical assumptions on the family of meshes on $\Omega$ (we refer the reader \cite{17} for details): there exist constants $c_0, c_1, c_2, c_3$ and $c_4$, independent of any given mesh in the family, such that the following hold
\begin{itemize}
\item For any given mesh $\mct$ in the family, let $p_{\mct}$ denotes the greatest number of elements to which any of the nodes belongs. Then
\begin{align*}
p_{\mct} \le c_0.
\end{align*}

\item For any triangle (or element) $K\in\mct$ with area $R_K$,
\begin{align*}
\frac{c_1}{N} \le R_K \le \frac{c_2}{N}.
\end{align*}

\item For any triangle in the given mesh with diameter $h_K$,
\begin{align*}
\frac{c_3}{N^{\frac{1}{2}}} \le h_K \le \frac{c_4}{N^{\frac{1}{2}}}.
\end{align*}

\end{itemize}
Also assume that $\{ \phi_1, ..., \phi_N\}$ are the standard basis functions for the conforming $H^1$-finite element space $V_N$, that is 
\begin{equation}\label{controlbound}
V_N = Span\{ \phi_1, ..., \phi_N\} \subset H_0^1(\Omega).
\end{equation}
The restriction of any basis function $\phi_i(x,y), i =1,2, ... ,N$ to any element $K\in \mct$ is a polynomial on $K$, i.e. $\phi_i(x,y), i=1,2,...,N$ is a piecewise polynomial in $\bar\Omega$. Also, if $\{ (x_i, y_i) \}_{i=1}^N$ are the nodes of $\mct$, then $\{ \phi_1, ..., \phi_N\}$ can be arranged such that $\phi_i(x_j,y_j) = \delta_{ij}, i, j = 1,2, ... ,N$. Define the following positive definite symmetric matrices 
\begin{equation}\label{FEMM}
\textit{(Mass)}\; M_N =  \begin{bmatrix}
(\phi_1, \phi_1) & \cdots & (\phi_1, \phi_N)\\
\vdots & & \vdots \\
(\phi_N, \phi_1) & \cdots & (\phi_N, \phi_N)
\end{bmatrix},
\end{equation}
\begin{equation}\label{FESM}
\textit{(Stiffness)}\; S_N =  \begin{bmatrix}
(\nab\phi_1, \nab\phi_1) & \cdots & (\nab\phi_1, \nab\phi_N)\\
\vdots & & \vdots \\
(\nab\phi_N, \nab\phi_1) & \cdots & (\nab\phi_N, \nab\phi_N)
\end{bmatrix}.
\end{equation}
Then the FEM approximating matrix to the generator 
\begin{equation}\label{SDG}
\mathcal{\overline{A}} = \begin{bmatrix}
0 & A  \\
-A & -\rho A 
\end{bmatrix}
\end{equation}
of the system \eqref{2ndODE} is given by
\begin{equation}\label{FEMSDG}
\mathcal{A}_{FE, N} = \begin{bmatrix}
\bigzero_N & M_N^{-1} S_N  \\
-M_N^{-1} S_N & -\rho M_N^{-1} S_N 
\end{bmatrix}
\end{equation}
where $\bigzero_N$ is the $N\times N$ zero matrix. Given arbitrary $[f,g]\in \bbR^{2N}$ and $\zeta\in L^2(0, T; \bbR^N)$, if we set 
\begin{equation}
\label{ASOLREPFEM}
\begin{bmatrix}
\xi(t) \\
\tilde{\xi}(t)
\end{bmatrix}=e^{\mathcal{A}_{FE, N}t}\begin{bmatrix}
f \\
g
\end{bmatrix}+\int_0^t \ {e^{\mathcal{A}_{FE ,N}(t-s)}\begin{bmatrix}
0 \\
\zeta(s)
\end{bmatrix}}\,ds,
\end{equation}
then the variables $[\xi(t), \tilde{\xi}(t)]$ satisfy the following ODE system:
\begin{subequations}
\label{FEMMainNCq}
\begin{alignat}{2}
\label{FEMMainNC1}
&\xi^{'}(t) = M_N^{-1}S_N\tilde{\xi}(t),\\
\label{FEMMainNC2}
&{\tilde\xi}^{'}(t) = -M_N^{-1}S_N\xi(t) -\rho M_N^{-1}S_N\tilde{\xi}(t) + \zeta(t),\\
\label{FEMMainNC3}
&[\xi(0), \tilde{\xi}(0)]  = [f,g]\in \bbR^{2N}.
\end{alignat}
\end{subequations}
Observe that \eqref{FEMMainNCq} is equivalent to the semidiscrete variational formulation of \eqref{2ndODE}. That is,
\begin{subequations}
\label{FEMVF}
\begin{alignat}{2}
\label{FEMVF1}
(v_N^{'}(t), \psi_N) &= (\nab w_N(t), \nab\psi_N),\quad \forall \psi_N\in V_N,\\
\label{FEMVF2}
(w_N^{'}(t), \varphi_N) &= -(\nab v_N(t), \nab\varphi_N) -\rho(\nab w_N(t), \nab\varphi_N) + (u_N(t), \varphi_N) ,\quad \forall \varphi_N\in V_N,\\
\label{FEMVF3}
[v_N(0), w_N(0)]  &= [v_{0,N}, w_{0,N}]\in V_N \times V_N,
\end{alignat}
\end{subequations}
where
\begin{align*}
v_N(t) = \sum_{i=1}^{N} \xi_i(t)\phi_i; \quad w_N(t) = \sum_{i=1}^{N} {\tilde{\xi}}_i(t)\phi_i; \quad u_N(t) = \sum_{i=1}^{N} {\zeta}_i(t)\phi_i,
\end{align*}
and
\begin{align*}
v_{0,N} = \sum_{i=1}^{N} f_i\phi_i; \quad w_{0,N} = \sum_{i=1}^{N} g_i\phi_i.
\end{align*}
 The following Theorem for the approximating system \eqref{FEMVF} was given in \cite[Theorem 4]{NCGAP}:
\begin{theorem}\label{thm:FEM} 
Suppose the classical mesh assumptions above are in place. Let also time $T>0$ be arbitrarily small. Then for the finite dimensional system \eqref{FEMVF} which approximates \eqref{2ndODE} there exists a sequence of null controllers $\{ u_N^* \} \subset L^2(0, T; \bbR^N)$, built upon the recipe provided in \cite{5}, for which the following estimate obtains, uniformly in N:
\begin{equation}\label{controlboundFEM1}
\Big( \int_0^T \ {\| u_N^*(t) \|^2} \,dt \Big)^{\frac{1}{2}} \le C T^{-\frac{3}{2}} \| [v_{0,N}, w_{0,N}]\|_{L^2(\Omega) \times L^2(\Omega)},
\end{equation}
where the constant $C$ is independent of N. \end{theorem}
\section{\textbf{Finite Difference Method (FDM) Approximation Scheme for \eqref {DiscRes1}}   }
\vspace{0.2cm}

\noindent \textbf{Application of Theorem \ref{thm:Discrete}}
 Let $\Omega = (0,a)^2$, $a>0$, and $h = \frac{a}{n+1}$ for a positive integer $n$. Also, let $h_x \times h_y$ be the uniform grid of $\Omega$, where $h_x : 0 = x_0 < x_1 < \cdots < x_n < x_{n+1} = a$, and  $h_y : 0 = y_0 < y_1 < \cdots < y_n < y_{n+1} = a$.

 The finite difference method approximates the values of $v$ and $w$ in \eqref{2ndODE} at the grid points $\{ (x_i, y_j): i,j = 1, ..., n\}$. In particular, we use central difference formula to discritize the spatial derivatives in \eqref{2ndODE} to get 
\begin{subequations}
\label{FDF}
\begin{alignat}{2}
\label{FDF1}
v_{i,j}^{'} &= \frac{4w_{i,j} - w_{i-1,j} - w_{i+1,j} - w_{i,j-1} - w_{i,j+1}}{h^2},\\
\label{FDF2}
w_{i,j}^{'} &= \frac{-4v_{i,j} + v_{i-1,j} + v_{i+1,j} + v_{i,j-1} + v_{i,j+1}}{h^2} \\
&\phantom{{}=1}+ \rho \frac{-4w_{i,j} + w_{i-1,j} + w_{i+1,j} + w_{i,j-1} + w_{i,j+1}}{h^2} + u_{i,j},\notag
\end{alignat}
\end{subequations}
where $v_{i,j}, w_{i,j}, u_{i,j}$ are the approximations of $v, w, u$ at $(x_i, y_j),$ respectively. With respect to the finite difference (FDM) scheme, the FDM approximating matrix to the generator
\begin{equation}\label{SDG2}
\mathcal{\overline{A}} = \begin{bmatrix}
0 & A  \\
-A & -\rho A 
\end{bmatrix}
\end{equation}
of the system \eqref{2ndODE} is given as 
\begin{figure}
\captionsetup{justification=centering}
\begin{tikzpicture}[bullet/.style={circle,fill=#1,inner sep=1.5pt}, scale = 0.85] 
 \draw (0,0) grid (7,7) foreach \X in {1,2,3,4,5,6} {foreach \Y in {1,2,3,4,5,6}
 {
  (\X,\Y) node[bullet=black]{}
 }}
 (0,7.5) -- (0,-0.5)
 (-0.5,0) -- (7.5,0)
 foreach \X in {-1,,+1} {(-0.6,\X+2) node[left]{}
 (\X+3,-0.6) node[below]{}};
\end{tikzpicture}
\caption{Finite difference grid $h_x \times h_y$ with $n=6$.}
\end{figure}
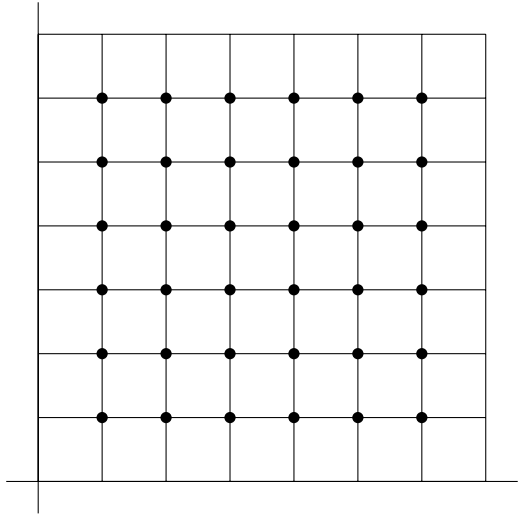
\label{FDD}
\begin{equation}\label{FDMSDG}
\mathcal{A}_{FD ,N} = \begin{bmatrix}
\bigzero_N & D_N  \\
-D_N & -\rho D_N 
\end{bmatrix},
\end{equation}
where $N = n^2$ and $D_N$ is the $N \times N$ block matrix given by
\begin{equation}\label{MBN}
D_N = \frac{1}{h^2} \left[\begin{array}{c|c|c|c|c}
F_n & -I_n & \bigzero_n & \cdots& \bigzero_n\\
\hline
  -I_n & \ddots & \ddots &\ddots & \vdots\\
\hline
\bigzero_n & \ddots &\ddots& \ddots & \bigzero_n\\
\hline
\vdots &\ddots& \ddots & \ddots  & -I_n \\
\hline
\bigzero_n & \cdots & \bigzero_n &  -I_n & F_n 
\end{array}\right].
\end{equation}
Here, $I_n$ and $\bigzero_n$ are the $n\times n$ identity and zero matrices, respectively, and $F_n$ is the $n\times n$ matrix given by
\begin{equation*}
F_n = \left[\begin{array}{c|c|c|c|c}
4 & -1 & 0 & \cdots& 0\\
\hline
  -1 & \ddots & \ddots &\ddots & \vdots\\
\hline
0 & \ddots &\ddots& \ddots & 0\\
\hline
\vdots &\ddots& \ddots & \ddots  & -1 \\
\hline
0 & \cdots & 0 &  -1 & 4 
\end{array}\right].
\end{equation*}
Given arbitrary $[f,g]\in \bbR^{2N}$ and $\zeta\in L^2(0, T; \bbR^N)$, if we set 
\begin{equation}
\label{ASOLREPFDM}
\begin{bmatrix}
\xi(t) \\
\tilde{\xi}(t)
\end{bmatrix}=e^{\mathcal{A}_{FD ,N}t}\begin{bmatrix}
f \\
g
\end{bmatrix}+\int_0^t \ {e^{\mathcal{A}_{FD, N}(t-s)}\begin{bmatrix}
0 \\
\zeta(s)
\end{bmatrix}}\,ds,
\end{equation}
then the variables $[\xi(t), \tilde{\xi}(t)]$ satisfy the following ODE system:
\begin{subequations}
\label{FDMMainNC}
\begin{alignat}{2}
\label{FDMMainNC1}
&\xi^{'}(t) = D_N\tilde{\xi}(t),\\
\label{FDMMainNC2}
&{\tilde\xi}^{'}(t) = -D_N(\xi(t) + \rho\tilde{\xi}(t)) + \zeta(t),\\
\label{FDMMainNC3}
&[\xi(0), \tilde{\xi}(0)]  = [f,g]\in \bbR^{2N}.
\end{alignat}
\end{subequations}
Observe that \eqref{FDMMainNC} is equivalent to the semidiscrete finite difference scheme of \eqref{2ndODE}, that is $[\bv_N,\bw_N]$
\begin{subequations}
\label{FDMNC}
\begin{alignat}{2}
\label{FDMNC1}
&\bv_N^{'}(t) = D_N\bw_N(t),\\
\label{FDMNC2}
&\bw_N^{'}(t) = -D_N(\bv_N(t) + \rho\bw_N(t)) + \bu_N(t),\\
\label{FDMNC3}
&[\bv_N(0), \bw_N(0)]  = [\bv_{0,N},\bw_{0,N}]\in \bbR^{2N},
\end{alignat}
\end{subequations}
where
\begin{align*}
\bv_N(t) = \left[\begin{array}{c}
\xi_{1,1}(t)\\
\vdots\\
\xi_{1,n}(t)\\
\hline
\xi_{2,1}(t)\\
\vdots\\
\xi_{2,n}(t)\\
\hline
\vdots\\
\vdots\\
\hline
\xi_{n,1}(t)\\
\vdots\\
\xi_{n,n}(t)
\end{array}\right]; \; \bw_N(t) = \left[\begin{array}{c}
\tilde{\xi}_{1,1}(t)\\
\vdots\\
\tilde{\xi}_{1,n}(t)\\
\hline
\tilde{\xi}_{2,1}(t)\\
\vdots\\
\tilde{\xi}_{2,n}(t)\\
\hline
\vdots\\
\vdots\\
\hline
\tilde{\xi}_{n,1}(t)\\
\vdots\\
\tilde{\xi}_{n,n}(t)
\end{array}\right]; \; \bu_N(t) = \left[\begin{array}{c}
\zeta_{1,1}(t)\\
\vdots\\
\zeta_{1,n}(t)\\
\hline
\zeta_{2,1}(t)\\
\vdots\\
\zeta_{2,n}(t)\\
\hline
\vdots\\
\vdots\\
\hline
\zeta_{n,1}(t)\\
\vdots\\
\zeta_{n,n}(t)
\end{array}\right]; 
\end{align*}
\begin{align*}
\bv_{0,N} = \left[\begin{array}{c}
f_{1,1}(t)\\
\vdots\\
f_{1,n}(t)\\
\hline
f_{2,1}(t)\\
\vdots\\
f_{2,n}(t)\\
\hline
\vdots\\
\vdots\\
\hline
f_{n,1}(t)\\
\vdots\\
f_{n,n}(t)
\end{array}\right]; \;
\bw_{0,N} = \left[\begin{array}{c}
g_{1,1}(t)\\
\vdots\\
g_{1,n}(t)\\
\hline
g_{2,1}(t)\\
\vdots\\
g_{2,n}(t)\\
\hline
\vdots\\
\vdots\\
\hline
g_{n,1}(t)\\
\vdots\\
g_{n,n}(t)
\end{array}\right].
\end{align*}
Here, $\xi_{i,j}(t), \tilde{\xi}_{i,j}(t), \zeta_{i,j}(t), f_{i,j}(t), g_{i,j}(t)$ are the approximations of $\xi, \tilde{\xi}, \zeta, f, g$ at $(x_i, y_j, t)$, respectively. In the following Theorem, we state our first result which gives the existence of null controllers for the finite difference method (FDM) approximating system \eqref{FDMNC} that satisfies the required blow up estimate in Theorem \ref{thm:Blowuprate}.
\begin{theorem}\label{thm:FDM}
Let terminal time $T>0$ be arbitrarily small. Then for the finite dimensional system \eqref{FDMNC} which approximates \eqref{2ndODE} there exists a sequence of null controllers $\{ u_N^* \} \subset L^2(0, T; \bbR^N)$, built upon the recipe provided in \cite{5}, for which the following estimate obtains, uniformly in N:
\begin{equation}\label{controlboundFDM1}
\Big( \int_0^T \ {\| u_N^*(t) \|^2_{\bbR^N}} \,dt \Big)^{\frac{1}{2}} \le C T^{-\frac{3}{2}} \| [\bv_{0,N}, \bw_{0,N}]\|_{\bbR^{2N}},
\end{equation}
where the constant $C$ is independent of N. 

\end{theorem}
\begin{proof}
 Our proof hinges on showing that the hypotheses of Theorem \ref{thm:Discrete} are satisfied under the setting of finite difference (FDM) approximation scheme. 

 The Kalman matrix of the system \eqref{FDMNC} is defined as the $2\times 2$ block matrix
\begin{equation}\label{FDMIKMP}
\mathcal{K}_N =[\mathcal{B}_N,\mathcal{A}_{FD ,N}\mathcal{B}_N]=\begin{bmatrix}
\bigzero_N & D_N  \\
I_N & -\rho D_N 
\end{bmatrix}
\end{equation}
where $\mathcal{B}_N = \begin{bmatrix}
\bigzero_N \\
I_N
\end{bmatrix}$,  $\mathcal{A}_{FD ,N}$ is the FDM approximating matrix given in \eqref{FDMSDG}, $D_N$ is the matrix in \eqref{MBN}. In order to show that the requirements \textbf{(A1)-(A3)} of Theorem \ref{thm:Discrete} holds, we will give the proof in two steps:\\

\noindent\textbf{Step 1:} Appealing to the theory of invertibility of $2\times 2$ block matrices in \cite{blockmatrix}, we observe that the Kalman matrix $\mathcal{K}_N$ defined in \eqref{FDMIKMP} will be invertible provided that the matrix $D_N$ (see \eqref{MBN}) is invertible. Since it can easily be proved that $D_N$ is a symmetric positive definite matrix it will be invertible which also yields that $\mathcal{K}_N$ is invertible with inverse 
\begin{equation}
\mathcal{K}_N^{-1} = \begin{bmatrix}
\rho I_N & I_N  \\
D_N^{-1} & \bigzero_N 
\end{bmatrix}.
\end{equation}
Using the Invertible Matrix Theorem we also infer that $\mathcal{K}_N$ will have the full rank $2N$ which proves the first requirement  \textbf{(A1)} of Theorem \ref{thm:Discrete} with index $k=1.$ To show that the matrix norm of the inverse matrix $\mathcal{K}_N^{-1}$ has a uniform bound that is independent of $N$, we use the special characterization of the matrix $D_N$ (see \cite{FDM2D} for details)
\begin{equation}
D_N = \frac{1}{h^2} (I_n \otimes E_n +  E_n \otimes I_n),
\end{equation}
where $E_n$ is the $n \times n$ matrix defined as
\begin{equation*}
E_n = \left[\begin{array}{c|c|c|c|c}
2 & -1 & 0 & \cdots& 0\\
\hline
  -1 & \ddots & \ddots &\ddots & \vdots\\
\hline
0 & \ddots &\ddots& \ddots & 0\\
\hline
\vdots &\ddots& \ddots & \ddots  & -1 \\
\hline
0 & \cdots & 0 &  -1 & 2 
\end{array}\right].
\end{equation*}
The eigenvalues of $D_N$ \cite{FDM2D} are given as
\begin{equation*}
\{ \lambda_{i, j} = \frac{1}{h^2}\big(4 - 2\big(\cos\big(\frac{i\pi}{n+1}\big) + \cos\big(\frac{j\pi}{n+1}\big)\big)\big): 1\le i,j \le n\}.
\end{equation*}
It can be observed that $\lambda_{i, j} > 0$ for all $1 \le i,j \le n$, and the smallest eigenvalue for $D_N$ is 
\begin{equation*}
\lambda_{1, 1} = \frac{4}{h^2}\big(1 - \cos\big(\frac{\pi}{n+1}\big)\big) = \frac{8\sin^2\big(\frac{h\pi}{2a}\big)}{h^2}\to \frac{2\pi}{a^2} \quad \textit{as} \quad h \to 0,
\end{equation*}
which yields that the eigenvalues of the symmetric positive definite matrix $D_N^{-1}$ will be bounded above uniformly in $N$ and $$\| D_N^{-1} \| \le C,$$ where the constant $C$ is independent of $N$. Consequently, if $\begin{bmatrix}
x_1 \\
x_2
\end{bmatrix} \in \bbR^{2N}$, then we have that 
\begin{align*}
\norm{\mathcal{K}_N^{-1}\begin{bmatrix}
x_1 \\
x_2
\end{bmatrix}}^2_{\bbR^{2N}} &= \| \rho x_1 + x_2  \|^2_{\bbR^N} + \| D_N^{-1} x_1\|^2_{\bbR^N}\\
&\le \max(1, \rho)(\|x_1\|^2_{\bbR^N} + \|x_2\|^2_{\bbR^N}) + \| D_N^{-1} \|^2 \|x_1\|^2_{\bbR^N}\\
&\le \max(1, \rho)(\|x_1\|^2_{\bbR^N} + \|x_2\|^2_{\bbR^N}) + C^2 (\|x_1\|^2_{\bbR^N} + \|x_2\|^2_{\bbR^N})\\
&\le \tilde{C}(\|x_1\|^2_{\bbR^N} + \|x_2\|^2_{\bbR^N}) 
\end{align*}
where $\tilde{C} = 2\max(1,\rho, C^2)$ is independent of $N$. This finishes the proof of requirement \textbf{(A2)} in Theorem \ref{thm:Discrete}.\\

\noindent\textbf{Step 2: }Since the Kalman rank condition is satisfied with index $k=1$, in this step, we will show that there are constants $D_j$ ($j=0,1)$ which satisfy (uniformly in $N$) the following inequalities:
\begin{equation}\label{ANboundFD}
\| \mathcal{A}_{FD, N}^{j} e^{\mathcal{A}_{FD, N} t} \| \le \frac{D_j}{t^j},\qquad j = 0,1.
\end{equation}

\noindent We start with the case $j=0:$ For this, we will show that the operator $\mathcal{A}_{FD,N}$ is maximal
disipative:

\textbf{a)} \textbf{Dissipativity: }For $[f,g]\in 
\mathbb{R}
^{2N},$%
\begin{eqnarray*}
\left\langle \mathcal{A}_{FD,N}\left[ 
\begin{array}{c}
f \\ 
g%
\end{array}%
\right] ,\left[ 
\begin{array}{c}
f \\ 
g%
\end{array}%
\right] \right\rangle  &=&\left\langle D_{N}g,f\right\rangle -\left\langle
D_{N}f,g\right\rangle -\rho \left\langle D_{N}g,g\right\rangle  \\
&=&-\rho \left\Vert D_{N}^{1/2}g\right\Vert \leq 0
\end{eqnarray*}

\textbf{(b) Maximality:} Given $[f,g]\in 
\mathbb{R}
^{2N},$ we consider the equation%
\[
\lbrack \lambda I_{2N}-\mathcal{A}_{FD,N}]\left[ 
\begin{array}{c}
v_{N} \\ 
z_{N}%
\end{array}%
\right] =\left[ 
\begin{array}{c}
f \\ 
g%
\end{array}%
\right] .
\]%
This becomes%
\[
\lambda v_{N}-D_{N}z_{N}=f
\]%
\[
\lambda z_{N}+D_{N}v_{N}+\rho D_{N}z_{N}=g
\]%
which after applying $-D_{N}$ to the first equation, and multiplying the
second one by $\lambda $ gives  
\[
-\lambda D_{N}v_{N}+D_{N}^{2}z_{N}=-D_{N}f
\]%
\[
\lambda ^{2}z_{N}+\lambda D_{N}v_{N}+\rho \lambda D_{N}z_{N}=\lambda g
\]%
and we get%
\[
\lambda ^{2}z_{N}+D_{N}^{2}z_{N}+\rho \lambda D_{N}z_{N}=\lambda g-D_{N}f.
\]%
Since $Null(\lambda ^{2}I_{N}+D_{N}^{2}+\rho \lambda D_{N})$ is empty then%
\[
z_{N}=(\lambda ^{2}I_{N}+D_{N}^{2}+\rho \lambda D_{N})^{-1}[\lambda g-D_{N}f]
\]%
and 
\[
v_{N}=\frac{1}{\lambda }D_{N}(\lambda ^{2}I_{N}+D_{N}^{2}+\rho \lambda
D_{N})^{-1}[\lambda g-D_{N}f]+\frac{1}{\lambda }f
\]%
this finishes the maximality of $\mathcal{A}_{FD,N}.$ Since $\{e^{\mathcal{A}%
_{FD,N}t}\}_{t\in 
\mathbb{R}
}$ is a group of contractions, then 
\begin{equation}
\left\Vert e^{\mathcal{A}_{FD,N}t}\right\Vert \leq 1,\text{ \ \ for every }%
t>0,\text{ \ \ }n\in 
\mathbb{N}
\label{1}
\end{equation}%
and the required estimate for the case $j=0$ is obtained with the constant $D_0=1$.

\bigskip 

To proceed with the case  $j=1,$ given the initial data $[v_{0N,}z_{0N}]\in 
\mathbb{R}
^{2N},$ set 
\begin{equation}
\left[ 
\begin{array}{c}
v_{N}(t) \\ 
z_{N}(t)%
\end{array}%
\right] =e^{\mathcal{A}_{FD,N}t}\left[ 
\begin{array}{c}
v_{0N} \\ 
z_{0N}%
\end{array}%
\right] .  \label{2}
\end{equation}%
Then taking the first and second derivative of both sides give%
\begin{equation}
\frac{d}{dt}\left[ 
\begin{array}{c}
v_{N}(t) \\ 
z_{N}(t)%
\end{array}%
\right] =\mathcal{A}_{FD,N}\left[ 
\begin{array}{c}
v_{N}(t) \\ 
z_{N}(t)%
\end{array}%
\right] \text{ \ \ \ \ \ \ \ or \ \ \ \ \ \ \ }%
\begin{array}{c}
v_{Nt}=D_{N}z_{N} \\ 
z_{Nt}=-D_{N}v_{N}-\rho D_{N}z_{N}%
\end{array}
\label{3}
\end{equation}%
and%
\begin{equation}
\frac{d^{2}}{dt^{2}}\left[ 
\begin{array}{c}
v_{N}(t) \\ 
z_{N}(t)%
\end{array}%
\right] =\mathcal{A}_{FD,N}\left[ 
\begin{array}{c}
v_{N_{t}}(t) \\ 
z_{N_{t}}(t)%
\end{array}%
\right] \text{ \ \ \ \ \ \ \ or \ \ \ \ \ \ \ }%
\begin{array}{c}
v_{Ntt}=D_{N}z_{N_{t}} \\ 
z_{Ntt}=-D_{N}v_{N_{t}}-\rho D_{N}z_{N_{t}}%
\end{array}
\label{4}
\end{equation}%
If we multiply the both sides of (\ref{3})$_{1}$ and (\ref{3})$_{2}$ by $%
v_{N},$ $z_{N},$ respectively and integrate from $0$ to $t$ we get 
\[
\int\limits_{0}^{t}\left\langle \left[ 
\begin{array}{c}
v_{Nt} \\ 
z_{Nt}%
\end{array}%
\right] ,\left[ 
\begin{array}{c}
v_{N} \\ 
z_{N}%
\end{array}%
\right] \right\rangle ds=\int\limits_{0}^{t}\left\langle \left[ 
\begin{array}{c}
D_{N}z_{N} \\ 
-D_{N}v_{N}-\rho D_{N}z_{N}%
\end{array}%
\right] ,\left[ 
\begin{array}{c}
v_{N} \\ 
z_{N}%
\end{array}%
\right] \right\rangle ds
\]%
or for $t>0,$%
\begin{equation}
\frac{1}{2}[\left\Vert v_{N}(t)\right\Vert ^{2}+\left\Vert
z_{N}(t)\right\Vert ^{2}]+\rho \int\limits_{0}^{t}\left\Vert
D_{N}^{1/2}z_{N}\right\Vert ^{2}ds=\frac{1}{2}[\left\Vert v_{0N}\right\Vert
^{2}+\left\Vert z_{0N}\right\Vert ^{2}]  \label{5}
\end{equation}%
Now, if we multiply (\ref{4})$_{1}$ by $t^{2}D_{N}^{-1}z_{N_{tt}},$ and (\ref%
{4})$_{2}$ by $-t^{2}D_{N}^{-1}v_{N_{tt}}$, integrate from $0$ to $t,$ and
add the resulting relations we obtain%
\begin{equation}
\int\limits_{0}^{t}s^{2}\left\langle z_{N_{t}},z_{N_{tt}}\right\rangle
ds+\int\limits_{0}^{t}s^{2}\left\langle v_{N_{t}},v_{N_{tt}}\right\rangle
ds+\rho \int\limits_{0}^{t}s^{2}\left\langle
z_{N_{t}},v_{N_{tt}}\right\rangle ds=0  \label{6}
\end{equation}%
Integrating by parts the first two terms on LHS of (\ref{6}) yields%
\begin{equation}
\frac{t^{2}}{2}[\left\Vert v_{N_{t}}(t)\right\Vert ^{2}+\left\Vert
z_{N_{t}}(t)\right\Vert ^{2}]+\rho \int\limits_{0}^{t}s^{2}\left\langle
z_{N_{t}},v_{N_{tt}}\right\rangle ds=\int\limits_{0}^{t}s[\left\Vert
v_{N_{t}}\right\Vert ^{2}+\left\Vert z_{N_{t}}\right\Vert ^{2}]ds  \label{7}
\end{equation}%
Invoking the first equation in (\ref{4}) also gives%
\begin{equation}
\frac{t^{2}}{2}[\left\Vert v_{N_{t}}(t)\right\Vert ^{2}+\left\Vert
z_{N_{t}}(t)\right\Vert ^{2}]+\rho \int\limits_{0}^{t}s^{2}\left\Vert
D_{N}^{1/2}z_{N_{t}}\right\Vert ^{2}ds=\int\limits_{0}^{t}s[\left\Vert
v_{N_{t}}\right\Vert ^{2}+\left\Vert z_{N_{t}}\right\Vert ^{2}]ds  \label{8}
\end{equation}%
To deal with RHS of (\ref{8}), we multiply (\ref{4})$_{1}$ by $%
tD_{N}^{-2}z_{N_{tt}},$ and (\ref{4})$_{2}$ by $-tD_{N}^{-2}v_{N_{tt}}$,
integrate from $0$ to $t,$ and add the resulting relations to have%
\[
\int\limits_{0}^{t}s\left\langle z_{N_{t}},D_{N}^{-1}z_{N_{tt}}\right\rangle
ds+\int\limits_{0}^{t}s\left\langle
v_{N_{t}},D_{N}^{-1}v_{N_{tt}}\right\rangle ds+\rho
\int\limits_{0}^{t}s\left\langle z_{N_{t}},D_{N}^{-1}v_{N_{tt}}\right\rangle
ds=0
\]%
This gives via integration by parts, 
\[
\frac{t}{2}[\left\Vert D_{N}^{-1/2}v_{N_{t}}(t)\right\Vert ^{2}+\left\Vert
D_{N}^{-1/2}z_{N_{t}}(t)\right\Vert ^{2}]+\rho
\int\limits_{0}^{t}s\left\langle z_{N_{t}},D_{N}^{-1}v_{N_{tt}}\right\rangle
ds \]
\[=\frac{1}{2}\int\limits_{0}^{t}[\left\Vert
D_{N}^{-1/2}v_{N_{t}}\right\Vert ^{2}+\left\Vert
D_{N}^{-1/2}z_{N_{t}}\right\Vert ^{2}]ds
\]%
Using (\ref{4})$_{1}$ and (\ref{3}), we then have for $t>0$,%
\[
\frac{t}{2}[\left\Vert D_{N}^{-1/2}v_{N_{t}}(t)\right\Vert ^{2}+\left\Vert
D_{N}^{-1/2}z_{N_{t}}(t)\right\Vert ^{2}]+\rho
\int\limits_{0}^{t}s\left\Vert z_{N_{t}}\right\Vert ^{2}ds
\]%
\begin{equation}
=\frac{1}{2}\int\limits_{0}^{t}[\left\Vert D_{N}^{1/2}z_{N}\right\Vert
^{2}+\left\Vert D_{N}^{1/2}v_{N}+\rho D_{N}^{1/2}z_{N}\right\Vert ^{2}]ds
\label{9}
\end{equation}%
For the RHS of (\ref{9}), we use (\ref{3})$_{2}$ to have%
\begin{eqnarray*}
\int\limits_{0}^{t}\left\langle D_{N}v_{N},v_{N}\right\rangle ds
&=&-\int\limits_{0}^{t}\left\langle z_{N_{t}},v_{N}\right\rangle ds-\rho
\int\limits_{0}^{t}\left\langle D_{N}z_{N},v_{N}\right\rangle ds \\
&=&-\left\langle z_{N},v_{N}\right\rangle |_{0}^{t}-\rho
\int\limits_{0}^{t}\left\langle
D_{N}^{1/2}z_{N},D_{N}^{1/2}v_{N}\right\rangle ds
\end{eqnarray*}%
Using (\ref{2}) and Young's inequality, we take%
\begin{eqnarray}
\int\limits_{0}^{t}\left\Vert D_{N}^{1/2}v_{N}\right\Vert ^{2}ds &\leq
&C\left\Vert [v_{0N,}z_{0N}]\right\Vert ^{2}+C_{\epsilon
}\int\limits_{0}^{t}\left\Vert D_{N}^{1/2}z_{N}\right\Vert ^{2}ds  \nonumber
\\
&\leq &C\left\Vert [v_{0N,}z_{0N}]\right\Vert ^{2},  \label{10}
\end{eqnarray}%
after using (\ref{5}). Now, applying (\ref{5}) and (\ref{10}) to the RHS of (%
\ref{9}), we get%
\[
\frac{t}{2}[\left\Vert D_{N}^{-1/2}v_{N_{t}}(t)\right\Vert ^{2}+\left\Vert
D_{N}^{-1/2}z_{N_{t}}(t)\right\Vert ^{2}]+\rho
\int\limits_{0}^{t}s\left\Vert z_{N_{t}}\right\Vert ^{2}ds
\]%
\begin{equation}
\leq C_{\rho }\left\Vert [v_{0N,}z_{0N}]\right\Vert ^{2}  \label{11}
\end{equation}%
Subsequently, if we multiply (\ref{4})$_{2}$ by $tD_{N}^{-1}v_{N_{t}}$ and
integrate in time we have%
\begin{eqnarray*}
\int\limits_{0}^{t}s\left\Vert v_{N_{t}}\right\Vert ^{2}ds
&=&-\int\limits_{0}^{t}s\left\langle
z_{N_{tt}},D_{N}^{-1}v_{N_{t}}\right\rangle ds-\rho
\int\limits_{0}^{t}s\left\langle z_{N_{t}},v_{N_{t}}\right\rangle ds \\
&=&-\int\limits_{0}^{t}s\left\langle
z_{N_{tt}},D_{N}^{-1}v_{N_{t}}\right\rangle ds-\rho
\int\limits_{0}^{t}s\left\langle z_{N_{t}},D_{N}z_{N}\right\rangle ds
\end{eqnarray*}%
after using (\ref{3})$_{1}$. If we also use (\ref{5}) and Young's Inequality
in the last relation we then obtain 
\begin{equation}
\int\limits_{0}^{t}s\left\Vert v_{N_{t}}\right\Vert ^{2}ds\leq \left\vert
\int\limits_{0}^{t}s\left\langle z_{N_{tt}},D_{N}^{-1}v_{N_{t}}\right\rangle
ds\right\vert +\epsilon \rho \int\limits_{0}^{t}s^{2}\left\Vert
D_{N}^{1/2}z_{N_{t}}\right\Vert ^{2}ds+C\left\Vert
[v_{0N,}z_{0N}]\right\Vert ^{2}  \label{12}
\end{equation}%
To handle the first term on RHS of (\ref{12}), integrating by parts we get%
\[
\int\limits_{0}^{t}s\left\langle z_{N_{tt}},D_{N}^{-1}v_{N_{t}}\right\rangle
ds=[s\left\langle D_{N}^{-1/2}z_{N_{t}},D_{N}^{-1/2}v_{N_{t}}\right\rangle
]|_{s=0}^{s=t}
\]%
\[
-\int\limits_{0}^{t}\left\langle
D_{N}^{-1/2}z_{N_{t}},D_{N}^{-1/2}v_{N_{t}}\right\rangle
ds-\int\limits_{0}^{t}s\left\langle
D_{N}^{-1/2}z_{N_{t}},D_{N}^{-1/2}v_{N_{tt}}\right\rangle ds
\]%
\[
=t\left\langle
D_{N}^{-1/2}z_{N_{t}}(t),D_{N}^{-1/2}v_{N_{t}}(t)\right\rangle
-\int\limits_{0}^{t}\left\langle
D_{N}^{-1/2}z_{N_{t}},D_{N}^{-1/2}v_{N_{t}}\right\rangle
ds-\int\limits_{0}^{t}s\left\Vert z_{N_{t}}\right\Vert ^{2}ds
\]%
After using (\ref{4})$_{1},$ 
\begin{equation}
t\left\langle D_{N}^{-1/2}z_{N_{t}}(t),D_{N}^{-1/2}v_{N_{t}}(t)\right\rangle
+\int\limits_{0}^{t}\left\langle D_{N}^{1/2}v_{N}+\rho
D_{N}^{1/2}z_{N},D_{N}^{1/2}z_{N}\right\rangle
ds-\int\limits_{0}^{t}s\left\Vert z_{N_{t}}\right\Vert ^{2}ds  \label{13}
\end{equation}%
Applying now the estimates (\ref{11}), (\ref{5}), and (\ref{10}) to RHS of (%
\ref{13}), we have%
\begin{equation}
\left\vert \int\limits_{0}^{t}s\left\langle
z_{N_{tt}},D_{N}^{-1}v_{N_{t}}\right\rangle ds\right\vert \leq C_{\rho
}\left\Vert [v_{0N,}z_{0N}]\right\Vert ^{2}  \label{14}
\end{equation}%
Now, using this estimate on the RHS of (\ref{12}), we get%
\begin{equation}
\int\limits_{0}^{t}s\left\Vert v_{N_{t}}\right\Vert ^{2}ds\leq \epsilon \rho
\int\limits_{0}^{t}s^{2}\left\Vert D_{N}^{1/2}z_{N_{t}}\right\Vert
^{2}ds+C\left\Vert [v_{0N,}z_{0N}]\right\Vert ^{2}  \label{15}
\end{equation}%
To conclude the proof of the case $j=1,$ we apply the estimates (\ref{11})
and (\ref{15}) to the RHS of (\ref{8}): this gives for $t>0,$ after taking $%
0<\epsilon <1,$%
\[
\frac{t^{2}}{2}[\left\Vert v_{N_{t}}\right\Vert ^{2}+\left\Vert
z_{N_{t}}\right\Vert ^{2}]+\rho (1-\epsilon
)\int\limits_{0}^{t}s^{2}\left\Vert D_{N}^{1/2}z_{N_{t}}\right\Vert
^{2}ds\leq C_{\rho }\left\Vert [v_{0N,}z_{0N}]\right\Vert ^{2}
\]%
Thus recalling (\ref{2}), we obtain%
\[
\left\Vert \frac{d}{dt}e^{\mathcal{A}_{FD,N}t}\left[ 
\begin{array}{c}
v_{0N} \\ 
z_{0N}%
\end{array}%
\right] \right\Vert \leq \frac{C_{\rho }}{t}\left\Vert \left[ 
\begin{array}{c}
v_{0N} \\ 
z_{0N}%
\end{array}%
\right] \right\Vert ,\text{ \ \ \ \ \ \ \ }\forall \text{ \ }\left[ 
\begin{array}{c}
v_{0N} \\ 
z_{0N}%
\end{array}%
\right] \in 
\mathbb{R}
^{2N}
\]%
which finishes the proof with the constant $D_1=C_{\rho}.$
\end{proof}
\begin{remark}
By means of a limiting process, it can be justified from Theorem \ref{thm:FEM} that there exists a null controller $u^*=lim_{n\rightarrow \infty}u_N^*$ to the elastic plate system that satisfies \eqref{controlboundFEM1}. Moreover, this control function will manifest the same asymptotics as that for the associated minimal energy function $\mathcal{E}_{min}(T).$ 
\end{remark}

\section{\textbf{Implementations of Numerical Schemes}}

This section is devoted to providing the algorithmic description of the finite element method (FEM) and finite difference method (FDM) schemes applied mainly on the system \eqref{2ndODE} or the finite dimensional systems \eqref{FEMVF} and \eqref{FDMNC}, respectively. We start with the FEM approximations.
\subsection{Implementation of the finite element method (FEM)}

 Approximating solutions to \eqref{2ndODE}, using the finite element method will require time discretization of the variational formulation \eqref{FEMVF}. For this, let $\Del t > 0$ be a given time step and assume that $u_N^{j+1}\in V_N$ represents an approximation of $u_N^*(t)$ at $t = t_{j+1} := (j+1)\Del t$. Then the fully-discrete scheme of \eqref{FEMVF} reads: for $j = 0, 1, 2, ...$, let $v_N^j, w_N^j, u_N^{j+1} \in V_N$ be given. Find $v_N^{j+1}, w_N^{j+1}$ such that

\begin{subequations}
\label{FEMVFTD}
\begin{alignat}{2}
\label{FEMVFTD1}
(v_N^{j+1}, \psi_N) &=\Del t (\nab w_N^{j+1}, \nab\psi_N) + (v_N^j, \psi_N),\quad \forall \psi_N\in V_N,\\
\label{FEMVFTD2}
(w_N^{j+1}, \varphi_N) &= (w_N^j, \varphi_N) - \Del t ((\nab v_N^{j+1}, \nab\varphi_N) + \rho(\nab w_N^{j+1}, \nab\varphi_N) - (u_N^{j+1}, \varphi_N)) ,\quad \forall \varphi_N\in V_N.
\end{alignat}
\end{subequations}
It is easy to show that \eqref{FEMVFTD} has a unique solution $v_N^{j+1}, w_N^{j+1}$ provided that $\Del t < \frac{1}{\rho}$, and this solution is the approximation to the solution of \eqref{2ndODE} at $t = t_{j+1}$. The crux of the computations is to compute the approximation to the null controller $u_N^{j+1}.$ With respect to the recipe given in \eqref{controlvec} to construct the approximate controllers, we remind the following notation:
In finite element method (FEM) approximation scheme, with respect to \eqref{FEMMainNCq}, the Kalman matrix $\mathcal{K}_N$ and its inverse $\mathcal{K}_N^{-1}$ are given by
\begin{equation}\label{FEMIKM}
\mathcal{K}_N =[\mathcal{B}_N,\mathcal{A}_{FE, N}]= \begin{bmatrix}
\bigzero_N & M_N^{-1} S_N  \\
I_N & -\rho M_N^{-1} S_N 
\end{bmatrix},\qquad
\mathcal{K}_N^{-1} = \begin{bmatrix}
\rho I_N & I_N  \\
S_N^{-1} M_N & \bigzero_N 
\end{bmatrix}
\end{equation}
where $\mathcal{B}_N = \begin{bmatrix}
\bigzero_N \\
I_N
\end{bmatrix},$
$\mathcal{A}_{FE, N}$ is the FEM approximation matrix (see \eqref{FEMSDG}) to the generator defined in \eqref{SDG}, and $M_N, S_N$ are the mass and stiffness matrices defined in \eqref{FEMM} and \eqref{FESM}, respectively. With the above notation now, referring to the formula \eqref{controlvec} for the construction of approximate controllers, we use the following:

Taking $\mathcal{A}_{N}=\mathcal{A}_{FE, N}$ as the FEM approximation matrix, the scalar valued function $f_T(t)$ as 
\begin{equation}\label{scalarf1}
f_T(t) = \frac{t^k(T-t)^k}{C_{T,k}}, \qquad C_{T,k} = \int_0^T \ {t^k(T-t)^k} \,dt,
\end{equation}
and 
\begin{equation}\label{controlvec1}
\mu_N(t) = \begin{bmatrix}
\mu_0(t) \\
\mu_1(t) \\
\mu_2(t) \\
\vdots      \\
\mu_k(t) 
\end{bmatrix}= -\mathcal{K}_N^{-1} e^{\mathcal{A}_{FE, N} t} \begin{bmatrix}
v_{0,N} \\
w_{0,N}
\end{bmatrix} f_T(t), \qquad 0\le t \le T,
\end{equation}
where each component $\mu_j$ is an $N-$vector, we have then the approximate controllers 
\begin{equation}
u_N^*(t) = \mu_0(t) + \mu_1^{'}(t) + \mu_2^{''}(t) + \cdots + \mu_k^{(k)}(t).
\end{equation}
We know that $e^{\mathcal{A}_{FE, N} t} \begin{bmatrix}
v_{0,N} \\
w_{0,N}
\end{bmatrix}$ represents the solution to the homogeneous variational formulation \eqref{FEMVF} (without the null controller term). That is, $$e^{\mathcal{A}_{FE, N} t} \begin{bmatrix}
v_{0,N} \\
w_{0,N}
\end{bmatrix} = \begin{bmatrix}
v_{N,h}(t) \\
w_{N,h}(t)
\end{bmatrix}$$ where $v_{N,h}(t), w_{N,h}(t) \in V_N$ satisfies (for all $t>0$):

\begin{subequations}
\label{FEMHVF}
\begin{alignat}{2}
\label{FEMHVF1}
(v_{N,h}^{'}(t), \psi_N) &= (\nab w_{N,h}(t), \nab\psi_N),\quad \forall \psi_N\in V_N,\\
\label{FEMHVF2}
(w_{N,h}^{'}(t), \varphi_N) &= -(\nab v_{N,h}(t), \nab\varphi_N) -\rho(\nab w_{N,h}(t), \nab\varphi_N),\quad \forall \varphi_N\in V_N,\\
\label{FEMHVF3}
[v_{N,h}(0), w_{N,h}(0)]  &= [v_{0,N}, w_{0,N}]\in V_N \times V_N.
\end{alignat}
\end{subequations}
To approximate $v_{N,h}(t), w_{N,h}(t)$ in \eqref{FEMHVF} at $t = t_{j+1}$ , we discretize \eqref{FEMHVF} in time with the same time stepping $\Del t$ used in \eqref{FEMVFTD} to get the following variational formulation:\\ For $j = 0, 1, 2, ...$, let $v_{N,h}^j, w_{N,h}^j \in V_N$ be given. Find $v_{N,h}^{j+1}, w_{N,h}^{j+1}$ such that

\begin{subequations}
\label{FEMHVFTD}
\begin{alignat}{2}
\label{FEMHVFTD1}
(v_{N,h}^{j+1}, \psi_N) &=\Del t (\nab w_{N,h}^{j+1}, \nab\psi_N) + (v_{N,h}^j, \psi_N),\quad \forall \psi_N\in V_N,\\
\label{FEMHVFTD2}
(w_{N,h}^{j+1}, \varphi_N) &= (w_{N,h}^j, \varphi_N) - \Del t ((\nab v_{N,h}^{j+1}, \nab\varphi_N) + \rho(\nab w_{N,h}^{j+1}, \nab\varphi_N)) ,\quad \forall \varphi_N\in V_N.
\end{alignat}
\end{subequations}
Then, by the above setting, we get

\begin{equation}\label{controlvecFEM}
\mu_N(t) = \begin{bmatrix}
\mu_0(t) \\
\mu_1(t) 
\end{bmatrix} = \begin{bmatrix}
-\rho I_N & -I_N  \\
-S_N^{-1} M_N & \bigzero_N 
\end{bmatrix} \begin{bmatrix}
v_{N,h}(t) \\
w_{N,h}(t)
\end{bmatrix} f_T(t) = \begin{bmatrix}
-(\rho v_{N,h}(t) + w_{N,h}(t))f_T(t) \\
-S_N^{-1} M_N v_{N,h}(t)f_T(t)
\end{bmatrix},
\end{equation}
where 
\begin{equation*}
f_T(t) = \frac{6t(T-t)}{T^3}, 
\end{equation*}
and $T$ is a given terminal time. Since $$u_N^*(t) =\mu_0(t) + \mu^{'}_1(t),$$ we turn our attention to approximate $\mu_0(t)$ and $\mu^{'}_1(t)$ at $t = t_{j+1},~~  j = 0, 1, 2, ...$. We approximate $\mu_0(t)$ at $t = t_{j+1}$ by

\begin{equation}\label{controlmu0FEM}
\mu_0(t_{j+1})\approx\mu_{0,N}^{j+1} := -(\rho v_{N,h}^{j+1} + w_{N,h}^{j+1})f_T(t_{j+1}).
\end{equation}
Since 
\begin{equation*}
\mu_1^{'}(t) = -S_N^{-1} M_N(v_{N,h}^{'}(t)f_T(t) + v_{N,h}(t)f_T^{'}(t)), 
\end{equation*}
then for a fixed $t>0$, $\mu_1^{'}(t)$ can be understood as the solution to the following variational formulation: Find $\mu_1^{'}(t)\in V_N$ such that
\begin{equation}
\label{FEMNCA}
(\nab \mu_1^{'}(t), \nab\psi_N) =- (G(t), \psi_N),\quad \forall \psi_N\in V_N,
\end{equation}
where 
\begin{equation*}
G(t) = v_{N,h}^{'}(t)f_T(t) + v_{N,h}(t)f_T^{'}(t).
\end{equation*}
 Because we are interested in approximating $\mu_1^{'}(t)$ at $t = t_{j+1},~~ j = 0, 1, 2, ...$, we approximate $G(t)$ at $t = t_{j+1}$ by
\begin{equation}
\label{FEMNCMU1}
G(t_{j+1}) \approx G_N^{j+1} := \frac{(v_{N,h}^{j+2} - v_{N,h}^{j+1})}{\Del t}f_T(t_{j+1}) + v_{N,h}^{j+1}f_T^{'}(t_{j+1}),
\end{equation}
hence, we approximate $\mu_1^{'}(t)$ at $t = t_{j+1}$ by $(\mu_{1,N}^{j+1})'$, where $(\mu_{1,N}^{j+1})'$ solves the following variational formulation : Find $(\mu_{1,N}^{j+1})'\in V_N$ such that 
\begin{equation}
\label{FEMNCAMU1}
(\nab (\mu_{1,N}^{j+1})', \nab\psi_N) =- (G_N^{j+1}, \psi_N),\quad \forall \psi_N\in V_N.
\end{equation}
Finally, we take $u_N^{j+1} := \mu_{0,N}^{j+1} + (\mu_{1,N}^{j+1})'$ to be the approximation of the null controller at $t = t_{j+1}$ and use it in \eqref{FEMVFTD}. 
Now, we provide an algorithm to summarize our implementation of the finite element method to solve \eqref{2ndODE}:

\begin{mathframed}
\smallskip\\
{\bf Algorithm 1 :}\label{Alg.1}
Let $T > 0$ (terminal time), $m \in \bbN$ ($m\ge 2$ is number of time stepping), and $\rho>2$ be user selected. Set $\Del t = \frac{T}{m}$, and $[v_N^{0}, w_N^{0}] = [v_{N,h}^{0}, w_{N,h}^{0}] =  [v_{0,N}, w_{0,N}]$. Then for $j = 0, 1, 2, ..., m-1$:
\begin{enumerate}
\item \textbf{Construction of $u_N^{j+1}$ :} Solve \eqref{FEMHVFTD} to find a solution $[v_{N,h}^{j+1}, w_{N,h}^{j+1}]$ and then use it again in \eqref{FEMHVFTD} to find $[v_{N,h}^{j+2}, w_{N,h}^{j+2}]$, that is: Find $[v_{N,h}^{j+2}, w_{N,h}^{j+2}] \in V_N \times V_N$ such that 
\begin{subequations}
\label{FEMHVFTDALG}
\begin{alignat}{2}
\label{FEMHVFTDALG1}
(v_{N,h}^{j+2}, \psi_N) &=\Del t (\nab w_{N,h}^{j+2}, \nab\psi_N) + (v_{N,h}^{j+1}, \psi_N),\quad \forall \psi_N\in V_N,\\
\label{FEMHVFTDALG2}
(w_{N,h}^{j+2}, \varphi_N) &= (w_{N,h}^{j+1}, \varphi_N) - \Del t ((\nab v_{N,h}^{j+2}, \nab\varphi_N) + \rho(\nab w_{N,h}^{j+2}, \nab\varphi_N)) ,\quad \forall \varphi_N\in V_N.
\end{alignat}
\end{subequations}
Set 
\begin{subequations}\label{MU0ADA}
\begin{alignat}{2}\label{MU0ADA1}
\mu_{0,N}^{j+1} &= -(\rho v_{N,h}^{j+1} + w_{N,h}^{j+1})f_T(t_{j+1}),\\
\label{MU0ADA2}
G_N^{j+1} &= \frac{(v_{N,h}^{j+2} - v_{N,h}^{j+1})}{\Del t}f_T(t_{j+1}) + v_{N,h}^{j+1}f_T^{'}(t_{j+1}).
\end{alignat}
\end{subequations}
Use $G_N^{j+1}$ (obtained in \eqref{MU0ADA2}) to find $(\mu_{1,N}^{j+1})'\in V_N$ by solving the variational formulation
\begin{equation}
\label{FEMNCAMU1}
(\nab (\mu_{1,N}^{j+1})', \nab\psi_N) =- (G_N^{j+1}, \psi_N),\quad \forall \psi_N\in V_N.
\end{equation}
Then set 
\begin{equation}\label{NCALG}
u_N^{j+1} = \mu_{0,N}^{j+1} + (\mu_{1,N}^{j+1})'.
\end{equation}

\item \textbf{Find $[v_N^{j+1}, w_N^{j+1}]$:} Use $u_N^{j+1}$ (obtained in \eqref{NCALG}) to find $[v_N^{j+1}, w_N^{j+1}]$ by solving the variational formulation \eqref{FEMVFTD}, that is: Find $[v_N^{j+1}, w_N^{j+1}]$ such that $\forall ~[\psi_N, \varphi_N] \in V_N \times V_N,$ 
\begin{subequations}
\label{FEMVFTDALG}
\begin{alignat}{2}
\label{FEMVFTDALG1}
(v_N^{j+1}, \psi_N) &=\Del t (\nab w_N^{j+1}, \nab\psi_N) + (v_N^j, \psi_N),\\
\label{FEMVFTDALG2}
(w_N^{j+1}, \varphi_N) &= (w_N^j, \varphi_N) - \Del t ((\nab v_N^{j+1}, \nab\varphi_N) + \rho(\nab w_N^{j+1}, \nab\varphi_N) - (u_N^{j+1}, \varphi_N)).
\end{alignat}
\end{subequations}
\end{enumerate}
\vspace{0.2cm}
\end{mathframed}
\subsection{Implementation of the Finite Difference Method (FDM)}

Similar to the FEM implementation, approximating solutions to \eqref{2ndODE} using the finite difference method will require time discretization of the finite difference scheme \eqref{FDMNC}.  Given a time step $\Del t > 0$ assume that $\bu_N^{j+1}\in \bbR^N$ is the vector whose components represent the approximation of $u_N^*(t)$ at $t = t_{j+1}$ and the grid points $(x_i, y_j)$ as labelled in \eqref{FDMNC}. Then the fully-discrete scheme of \eqref{FDMNC} reads: For $j = 0, 1, 2, ...$, let $\bv_N^j, \bw_N^j, \bu_N^{j+1} \in \bbR^N$ be given. Find $\bv_N^{j+1}, \bw_N^{j+1} \in \bbR^N$ such that
\begin{subequations}
\label{FDMMainNCTD}
\begin{alignat}{2}
\label{FEDMainNCTD1}
&\bv_N^{j+1} - \Del t D_N\bw_N^{j+1} = \bv_N^j,\\
\label{FDMMainNCTD2}
&\bw_N^{j+1} + \Del t D_N(\bv_N^{j+1} + \rho\bw_N^{j+1}) = \Del t \bu_N^{j+1} + \bw_N^j.
\end{alignat}
\end{subequations}
The solution $\bv_N^{j+1}, \bw_N^{j+1}$ to \eqref{FDMMainNCTD} are the vectors whose components represent the approximation to the solution of \eqref{2ndODE} at $t = t_{j+1}$ and the grid points $(x_i, y_j)$ as labelled in \eqref{FDMNC}.
Observe that \eqref{FDMMainNCTD} can be written as a linear $2 \times 2$ block system $\bf{A}\bf{x} = \bf{b}$, where 

\[
\bf{A} = \left[\begin{array}{cc}
I_N & -\Del t D_N \\
\Del t D_N & I_N+ \rho\Del t D_N
\end{array}\right]; \qquad
\bf{x} = \left[\begin{array}{c}
\bv_N^{j+1}  \\
\bw_N^{j+1}
\end{array}\right]; \qquad
\bf{b} = \left[\begin{array}{c}
\bv_N^j\\
\Del t \bu_N^{j+1} + \bw_N^j
\end{array}\right].
\]
The system \eqref{FDMMainNCTD} has a unique solution if the $2N \times 2N$ matrix $\bf{A}$ is invertible. Since the Schur complement of $\bf{A}$ will be the matrix $I_N + \rho\Del t D_N + (\Del t D_N)^2$ which is invertible, appealing to the theory of $2\times 2$ matrices we infer that the matrix $\bf{A}$ is invertible. For a detailed discussion, we refer the reader to \cite{blockmatrix}.

\noindent Similar to the finite element scheme, the crux of the computations is to compute the approximation to the null controller $\bu_N^{j+1}.$With respect to the recipe given in \eqref{controlvec} to construct the approximate controllers, we define the following matrices:\\
In finite difference method (FDM) approximation scheme, with respect to \eqref{FDMMainNC}, the Kalman matrix $\mathcal{K}_N$ and its inverse $\mathcal{K}_N^{-1}$ can be computed explicitly in terms of the matrix $D_N$ defined in \eqref{MBN}:
\begin{equation}\label{FDMIKM}
\mathcal{K}_N =[\mathcal{B}_N, \mathcal{B}_N\mathcal{A}_{FD, N}]= \begin{bmatrix}
\bigzero_N & D_N  \\
I_N & -\rho D_N 
\end{bmatrix},\qquad
\mathcal{K}_N^{-1} = \begin{bmatrix}
\rho I_N & I_N  \\
D_N^{-1} & \bigzero_N 
\end{bmatrix}.
\end{equation}
Here $\mathcal{B}_N = \begin{bmatrix}
\bigzero_N \\
I_N
\end{bmatrix},$
and $\mathcal{A}_{FD, N}$ is the FDM approximation matrix \eqref{FDMSDG} to the generator defined in \eqref{SDG2}. With the above notation now, referring to the formula \eqref{controlvec} for the construction of approximate controllers, we take $\mathcal{A}_{N}=\mathcal{A}_{FD, N}$ as the FDM approximation matrix, the scalar valued function $f_T(t)$ as 
\begin{equation}\label{scalarf2}
f_T(t) = \frac{t^k(T-t)^k}{C_{T,k}}, \qquad C_{T,k} = \int_0^T \ {t^k(T-t)^k} \,dt,
\end{equation}
and 
\begin{equation}\label{controlvec2}
\mu_N(t) = \begin{bmatrix}
\mu_0(t) \\
\mu_1(t) \\
\mu_2(t) \\
\vdots      \\
\mu_k(t) 
\end{bmatrix}= -\mathcal{K}_N^{-1} e^{\mathcal{A}_{FD, N} t} \begin{bmatrix}
v_{0,N} \\
w_{0,N}
\end{bmatrix} f_T(t), \qquad 0\le t \le T.
\end{equation}
Observe that $e^{\mathcal{A}_{N} t} Y_{N,0}$ in \eqref{controlvec} becomes $e^{\mathcal{A}_{FD, N} t} \begin{bmatrix}
\bv_{0,N} \\
\bw_{0,N}
\end{bmatrix}$ in the finite difference setting, and it represents the solution to the finite difference scheme \eqref{FDMNC} without the null controller term. That is, $$e^{\mathcal{A}_{FD, N} t} \begin{bmatrix}
\bv_{0,N} \\
\bw_{0,N}
\end{bmatrix} = \begin{bmatrix}
\bv_{N,h}(t) \\
\bw_{N,h}(t)
\end{bmatrix},$$ where $\bv_{N,h}(t), \bw_{N,h}(t) \in \bbR^N$ satisfies (for all $t>0$):
\begin{subequations}
\label{FDMHNC}
\begin{alignat}{2}
\label{FDMHNC1}
&\bv_{N,h}^{'}(t) = D_N\bw_{N,h}(t),\\
\label{FDMHNC2}
&\bw_{N,h}^{'}(t) = -D_N(\bv_{N,h}(t) + \rho\bw_{N,h}(t)),\\
\label{FDMHNC3}
&[\bv_N(0), \bw_N(0)]  = [\bv_{0,N},\bw_{0,N}]\in \bbR^{2N}.
\end{alignat}
\end{subequations}
To approximate $\bv_{N,h}(t),\bw_{N,h}(t)$ in \eqref{FDMHNC} at $t = t_{j+1}$ , we discretize \eqref{FDMHNC} in time using the same time stepping $\Del t$ we used in \eqref{FDMMainNCTD} to get the following finite difference scheme: For $j = 0, 1, 2, ...$, let $\bv_{N,h}^j, \bw_{N,h}^j \in \bbR^N$ be given. Find $\bv_{N,h}^{j+1}, \bw_{N,h}^{j+1}\in \bbR^N$ such that
\begin{subequations}
\label{FDMHVFTD}
\begin{alignat}{2}
\label{FDMHVFTD1}
&\bv_{N,h}^{j+1} - \Del t D_N\bw_{N,h}^{j+1} = \bv_N^j,\\
\label{FDMHVFTD2}
&\bw_{N,h}^{j+1} + \Del t D_N(\bv_{N,h}^{j+1} + \rho\bw_{N,h}^{j+1}) = \bw_{N,h}^j.
\end{alignat}
\end{subequations}
Observe that the null control formula in the finite difference setting becomes 
\begin{equation*}
\mu(t) = \begin{bmatrix}
\mu_0(t) \\
\mu_1(t) 
\end{bmatrix} = \begin{bmatrix}
-\rho I_N & -I_N  \\
-D_N^{-1} & \bigzero_N 
\end{bmatrix} \begin{bmatrix}
\bv_{N,h}(t) \\
\bw_{N,h}(t)
\end{bmatrix} f_T(t)\end{equation*}
\begin{equation}\label{controlvecFEM}
=\begin{bmatrix}
-(\rho \bv_{N,h}(t) + \bw_{N,h}(t))f_T(t) \\
-D_N^{-1}\bv_{N,h}(t)f_T(t)
\end{bmatrix},
\end{equation}
where 
\begin{equation*}
f_T(t) = \frac{6t(T-t)}{T^3}, 
\end{equation*}
and $T$ is a given terminal time. Since $$u_N^*(t) =\mu_0(t) + \mu^{'}_1(t),$$ we turn our attention to approximate $\mu_0(t)$ and $\mu^{'}_1(t)$ at $t = t_{j+1},~~ j = 0, 1, 2, ...$ We approximate $\mu_0(t)$ at $t = t_{j+1}$ by
\begin{equation}\label{controlmu0FEM}
\mu_0(t_{j+1})\approx\bm{\mu}_{0,N}^{j+1} := -(\rho \bv_{N,h}^{j+1} + \bw_{N,h}^{j+1})f_T(t_{j+1}).
\end{equation}
Since 
\begin{equation*}
\mu_1^{'}(t) = -D_N^{-1}(\bv_{N,h}^{'}(t)f_T(t) + \bv_{N,h}(t)f_T^{'}(t)), 
\end{equation*}
then for a fixed $t>0$,  $\mu_1^{'}(t)$ can be visualized as the solution to the following finite difference scheme:
\begin{equation}
\label{FDMNCA}
D_N\mu_1^{'}(t) = -G(t),
\end{equation}
where 
\begin{equation*}
G(t) = \bv_{N,h}^{'}(t)f_T(t) + \bv_{N,h}(t)f_T^{'}(t).
\end{equation*}
 Since we are interested in approximating $\mu_1^{'}(t)$ at $t = t_{j+1}, ~~j = 0, 1, 2, ...$, we approximate $G(t)$ at $t = t_{j+1}$ by
\begin{equation}
\label{FDMNCMU1}
G(t_{j+1}) \approx \bm{G}_N^{j+1} := \frac{(\bv_{N,h}^{j+2} - \bv_{N,h}^{j+1})}{\Del t}f_T(t_{j+1}) + \bv_{N,h}^{j+1}f_T^{'}(t_{j+1}),
\end{equation}
Using \eqref{FDMNCMU1} we now approximate $\mu_1^{'}(t)$ at $t = t_{j+1}$ by $(\bm{\mu}_{1,N}^{j+1})'$, where $(\bm{\mu}_{1,N}^{j+1})'$ solves the following finite difference scheme:
\begin{equation}
\label{FDMNCAMU1}
D_N (\bm{\mu}_{1,N}^{j+1})' = -\bm{G}_N^{j+1} ,
\end{equation}
Finally, we take $\bu_N^{j+1} := \bm{\mu}_{0,N}^{j+1} + (\bm{\mu}_{1,N}^{j+1})'$ to be the approximation of the null controller at $t = t_{j+1}$ and use it in \eqref{FDMMainNCTD}. We provide an algorithm to summarize our implementation of the finite difference method to solve \eqref{2ndODE}:

\begin{mathframed}
\smallskip\\
{\bf Algorithm 2 :}\label{Alg.2}
Let $T > 0$ (terminal time), $m \in \bbN$ ($m\ge 2$ is number of time stepping), and $\rho>2$ be user selected. Set $\Del t = \frac{T}{m}$, and $[\bv_N^{0}, \bw_N^{0}] = [\bv_{N,h}^{0}, \bw_{N,h}^{0}] =  [\bv_{0,N}, \bw_{0,N}]$. Then for $j = 0, 1, 2, ..., m-1$:
\begin{enumerate}
\item \textbf{Construction of $\bu_N^{j+1}$:} Solve \eqref{FDMHVFTD} to find the solution $[\bv_{N,h}^{j+1}, \bw_{N,h}^{j+1}]$ and then use this solution again in \eqref{FDMHVFTD} to find $[\bv_{N,h}^{j+2}, \bw_{N,h}^{j+2}]$. That is,\\ find $[\bv_{N,h}^{j+2}, \bw_{N,h}^{j+2}] \in \bbR^N \times \bbR^N$ such that 
\begin{subequations}
\label{FDMHVFTDALG}
\begin{alignat}{2}
\label{FDMHVFTDALG1}
&\bv_{N,h}^{j+2} - \Del t D_N\bw_{N,h}^{j+2} = \bv_{N,h}^{j+1},\\
\label{FDMHVFTDALG2}
&\bw_{N,h}^{j+2} + \Del t D_N(\bv_{N,h}^{j+2} + \rho\bw_{N,h}^{j+2}) = \bw_{N,h}^{j+1}.
\end{alignat}
\end{subequations}
Set 
\begin{subequations}\label{MU0ADAFD}
\begin{alignat}{2}\label{MU0ADAFD1}
\bm{\mu}_{0,N}^{j+1} &= -(\rho \bv_{N,h}^{j+1} + \bw_{N,h}^{j+1})f_T(t_{j+1}),\\
\label{MU0ADAFD2}
\bm{G}_N^{j+1} &= \frac{(\bv_{N,h}^{j+2} - \bv_{N,h}^{j+1})}{\Del t}f_T(t_{j+1}) + \bv_{N,h}^{j+1}f_T^{'}(t_{j+1}).
\end{alignat}
\end{subequations}
Use $\bm{G}_N^{j+1}$ (obtained in \eqref{MU0ADAFD2}) to find $(\bm{\mu}_{1,N}^{j+1})'\in \bbR^N$ by solving 
\begin{equation}
\label{FDMNCAMU1ALG}
D_N(\bm{\mu}_{1,N}^{j+1})' = -\bm{G}_N^{j+1}.
\end{equation}
Then set 
\begin{equation}\label{FDNCALG}
\bu_N^{j+1} := \bm{\mu}_{0,N}^{j+1} + (\bm{\mu}_{1,N}^{j+1})'.
\end{equation}

\item \textbf{Find $[\bv_N^{j+1}, \bw_N^{j+1}]$:} Use $\bu_N^{j+1}$ (obtained in \eqref{FDNCALG}) to find $[\bv_N^{j+1}, \bw_N^{j+1}]$ by solving the system \eqref{FDMMainNCTD}. That is, find $[\bv_N^{j+1}, \bw_N^{j+1}]$ such that
\begin{subequations}
\label{FDMVFTDALG}
\begin{alignat}{2}
\label{FDMVFTDALG1}
&\bv_N^{j+1} - \Del t D_N\bw_N^{j+1} = \bv_N^j,\\
\label{FDMVFTDALG2}
&\bw_N^{j+1} + \Del t D_N(\bv_N^{j+1} + \rho\bw_N^{j+1}) = \Del t \bu_N^{j+1} + \bw_N^j.
\end{alignat}
\end{subequations}
\end{enumerate}
\vspace{0.2cm}
\end{mathframed}
\section{\textbf{Numerical Experiments}}\label{sec-5}
\noindent In this section, we perform some numerical experiments
and compare the results with the theoretical ones
given in the previous sections. We consider an example where the data
is taken to be 
$\Omega = (0,\pi)^2$, $\rho = \frac{5}{2}$ and the initial condition to \eqref{2ndODE} is given as
\begin{align}
\label{eqn:TestProblem}
{\begin{pmatrix}
v_0(x,y)\\
w_0(x,y)
\end{pmatrix}=\begin{pmatrix}
0 \\
\frac{3}{2}\sin(2x)\sin(2y)
\end{pmatrix}}
\end{align}
We use the exact solution to the homogeneous part of the system \eqref{2ndODE} which is derived in Section 7.
\vspace{0.2cm}

\subsection{\textbf{Finite element scheme}}
\vspace{0.2cm}

By the use of Algorithm 1, $$(v_{N, h}(t), w_{N, h}(t)) \approx (v_N(t), w_N(t))~~~ \text{and}~~~ u_h^*(t) \approx u_N(t),$$ in tables \ref{tableOne}, \ref{tableTwo}, and \ref{tableThree}, denote the computed solution pair and the null controller for \eqref{2ndODE}, respectively. The mesh size is taken to be $h = \frac{1}{32}$ (or $N = 3338$) on a Delaunay triangulation using continuous functions on $\mct$ that are polynomials of degree one when restricted to any element $K\in\mct$. 

Tables \ref{tableOne} and \ref{tableTwo} show that $(v_{N, h}(T), w_{N, h}(T)) \to 0$ when $T$ is relatively big. Recall that the formula in \eqref{controlvec} is an approximation to the control function that will lead the solution $(v_N(t), w_N(t)) \to (0,0).$

Table \ref{tableThree} shows that the computed null control obeys the blowup rate in Theorem \ref{thm:FEM} as $T\to 0$. Also, the logarithmic graph in Figure 2 shows that the blowup rate for the computed null control $u_h^*(t) $ is similar to the graph of $y = x^{\frac{-3}{2}}$.
\begin{center}
\begin{table}
\caption{\label{tableOne}Errors and rates of convergence
for example \eqref{eqn:TestProblem} with time step $\Delta t = 0.2$ using Algorithm 1.}
 \begin{tabular}{||c c c c c||} 
 \hline
 $T$ & $\|v_{N, h}(T)\|^2 + \|w_{N, h}(T)\|^2$ & rate & $\|u_h^*\|_{L^2(L^2(\Omega);0,T)}$ & rate \\ [0.5ex] 
 \hline
$2^{1}$ &5.6144E-02	&--		&2.8778E-01 	&--\\
$2^{2}$ &1.5294E-02	&1.876         	&8.0441E-02  	&1.838\\
$2^{3}$ &3.9255E-03	&1.962	&2.1203E-02 	&1.923\\
$2^{4}$ &9.9397E-04	&1.981	&5.4391E-03 	&1.962\\
$2^{5}$ &2.5006E-04	&1.991	&1.3771E-03 	&1.981\\
$2^{6}$ &6.2713E-05	&1.995	&3.4646E-04 	&1.991\\
 \hline
\end{tabular}
\end{table}
\end{center}
\begin{center}
\begin{table}
\caption{\label{tableTwo}Errors and rates of convergence
for example \eqref{eqn:TestProblem} with time step $\Delta t = 0.1$ using Algorithm 1.}
 \begin{tabular}{||c c c c c||} 
 \hline
 $T$ & $\|v_{N, h}(T)\|^2 + \|w_{N, h}(T)\|^2$ & rate & $\|u_h^*\|_{L^2(L^2(\Omega);0,T)}$ & rate \\ [0.5ex] 
 \hline
$2^{1}$ &4.2633E-02	&--		&3.0454E-01 	&--\\
$2^{2}$ &1.1209E-02	&1.927         	&8.4262E-02  	&1.854\\
$2^{3}$ &2.8385E-03	&1.981	&2.2117E-02 	&1.929\\
$2^{4}$ &7.1412E-04	&1.991	&5.6627E-03 	&1.966\\
$2^{5}$ &1.7909E-04	&1.995	&1.4324E-03 	&1.983\\
$2^{6}$ &4.4843E-05	&1.998	&3.6021E-04 	&1.992\\
 \hline
\end{tabular}
\end{table}
\end{center}

\begin{center}
\begin{table}
\caption{\label{tableThree}Errors and rates of convergence
for example \eqref{eqn:TestProblem} with time step $\Delta t = \frac{1}{1536}$ using Alg. 1.}
 \begin{tabular}{||c c c c c||} 
 \hline
 $T$ & $\|v_{N, h}(T)\|^2 + \|w_{N, h}(T)\|^2$ & rate & $\|u_h^*\|_{L^2(L^2(\Omega);0,T)}$ & rate \\ [0.5ex] 
 \hline
$2^{-4}$ &1.0363E 00	&--		&2.0955E+01 	&--\\
$2^{-5}$ &1.3295E 00	&-0.35         	&3.5071E+01  	&-0.74\\
$2^{-6}$ &1.5819E 00	&-0.25	&5.7895E+01	&-0.72\\
$2^{-7}$ &2.0669E 00	&-0.38	&1.0233E+02 	&-0.82\\
$2^{-8}$ &3.7593E 00	&-0.86	&2.1864E+02 	&-1.09\\
$2^{-9}$ &1.1112E+01	&-1.56	&6.2465E+02 	&-1.51\\
 \hline
\end{tabular}
\end{table}
\end{center}

\begin{center}
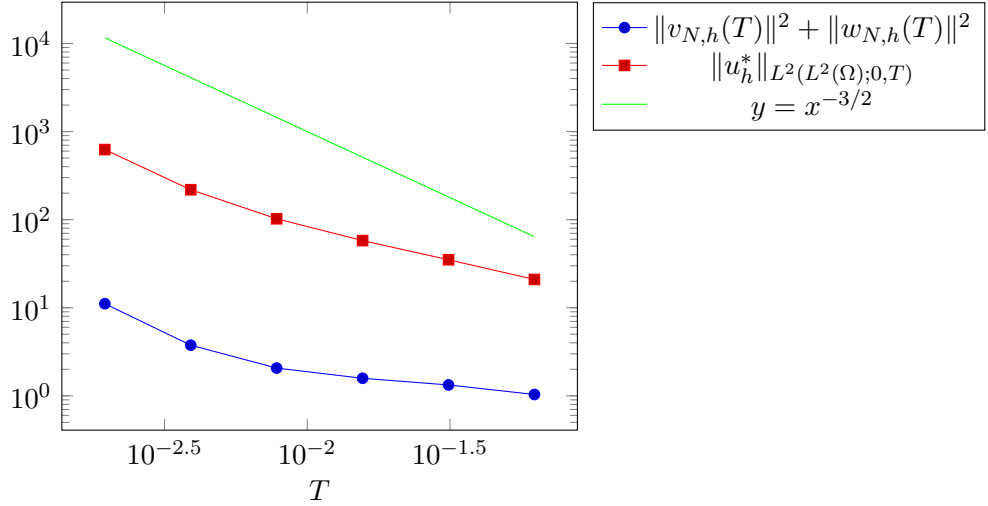
\begin{figure}
\begin{tikzpicture}
\begin{loglogaxis}[
	xlabel={$T$},
	ylabel={},
	legend pos=outer north east
]
\addplot coordinates {
	 (1/512,1.1112E+01 )   (1/256,  3.7593)
	(1/128, 2.0669)  (1/64,  1.5819)  (1/32,  1.3295)
	(1/16, 1.0363 ) 
};

\addplot coordinates{
	 (1/512,6.2465E+02 )   (1/256,  2.1864E+02)
	(1/128, 1.0233E+02)  (1/64,  5.7895E+01)  (1/32,  3.5071E+01)
	(1/16, 2.0955E+01) 
};
\addplot [
    domain=1/512:1/16, 
    samples=100, 
    color=green,
]
{x^-3/2};
\legend{$\|v_{N, h}(T)\|^2 + \|w_{N, h}(T)\|^2$,$\|u_h^*\|_{L^2(L^2(\Omega);0,T)}$, $y = x^{-3/2}$}
\end{loglogaxis}
\end{tikzpicture}
\caption{Logarithmic plots of $\|v_{N, h}(T)\|^2 + \|w_{N, h}(T)\|^2$ vs. $\|u_h^*\|_{L^2(L^2(\Omega);0,T)}$ vs.$y = x^{-3/2}$  using Alg. 1.}
\end{figure}\label{FEMGraph}
\end{center}
\vspace{-1.9cm}

\subsection{\textbf{Finite difference scheme}}

By using Algorithm 2, $$(v_{N, h}(t), w_{N, h}(t)) \approx (v_N(t), w_N(t))~~ \text{and}~~u_h^*(t) \approx u_N(t)$$ denote the computed solution pair and the null controller for \eqref{2ndODE}, respectively in tables \ref{tableFour}, \ref{tableFive}, and \ref{tableSix}. The grid size is taken to be $n = 32$. 

Tables \ref{tableFour} and \ref{tableFive} show that $(v_{N, h}(T), w_{N, h}(T)) \to 0$ as $T$ gets large. Recall that the formula in \eqref{controlvec} is an approximation to the control that will lead the solution $(v_N(t), w_N(t)) \to (0,0).$ 
 Table \ref{tableSix} shows that the computed null control fluctuates around the blowup rate in Theorem \ref{thm:FDM} as $T\to 0$. Also, the logarithmic graph in Figure 3 shows that the blowup rate for the computed null control $u_h^*(t) $ is similar to the graph of $y = x^{\frac{-3}{2}}$.
\begin{center}
\begin{table}
\caption{\label{tableFour}Errors and rates of convergence
for example \eqref{eqn:TestProblem} with time step $\Delta t = 0.2$ using Alg. 2.}
 \begin{tabular}{||c c c c c||} 
 \hline
 $T$ & $\|v_{N, h}\|^2_{\bbR^{N}} + \|w_{N, h}\|^2_{\bbR^{N}}$ & rate & $\|u_h^*\|_{(\bbR^2N; 0,T)}$ & rate \\ [0.5ex] 
 \hline
$2^{1}$ &4.7354E-06	&--		&2.1344E 00 	&--\\
$2^{2}$ &0.0000E 00	&--         	&5.8009E-01  	&1.879\\
$2^{3}$ &0.0000E 00	&--         	&1.5152E-01 	&1.937\\
$2^{4}$ &0.0000E 00	&--	           &3.8729E-02 	&1.968\\
$2^{5}$ &0.0000E 00	&--	           &9.7906E-03 	&1.984\\
$2^{6}$ &0.0000E 00	&--	           &2.4613E-03	&1.992\\
 \hline
\end{tabular}
\end{table}
\end{center}
\begin{center}
\begin{table}
\caption{\label{tableFive}Errors and rates of convergence
for example \eqref{eqn:TestProblem} with time step $\Delta t = 0.1$ using Alg. 2.}
 \begin{tabular}{||c c c c c||} 
 \hline
 $T$ & $\|v_{N, h}\|^2_{\bbR^{N}} + \|w_{N, h}\|^2_{\bbR^{N}}$ & rate & $\|u_h^*\|_{L^2(L^2(\Omega);0,T)}$ & rate \\ [0.5ex] 
 \hline
$2^{1}$ &3.9361E-06	&--		&3.6379E 00 	&--\\
$2^{2}$ &0.0000E 00	&--         	&9.5927E-01  	&1.923\\
$2^{3}$ &0.0000E 00	&--         	&2.4645E-01 	&1.961\\
$2^{4}$ &0.0000E 00	&--	           &6.2465E-02 	&1.980\\
$2^{5}$ &0.0000E 00	&--	           &1.5724E-02 	&1.990\\
$2^{6}$ &0.0000E 00	&--	           &3.9446E-03	&1.995\\
 \hline
\end{tabular}
\end{table}
\end{center}

\begin{center}
\begin{table}
\caption{\label{tableSix}Errors and rates of convergence
for example \eqref{eqn:TestProblem} with time step $\Delta t = \frac{1}{1536}$ using Alg. 2.}
 \begin{tabular}{||c c c c c||} 
 \hline
 $T$ & $\|v_{N, h}\|^2_{\bbR^{N}} + \|w_{N, h}\|^2_{\bbR^{N}}$ & rate & $\|u_h^*\|_{L^2(L^2(\Omega);0,T)}$ & rate \\ [0.5ex] 
 \hline
$2^{-4}$ &3.5527E+05	&--		           &8.9903E+05 	&--\\
$2^{-5}$ &1.5122E+06	&-2.090         	&2.8531E+06  	&-1.666\\
$2^{-6}$ &2.6687E+06	&-0.819         	&7.6917E+06	&-1.431\\
$2^{-7}$ &2.9605E+06	&-0.150         	&1.9956E+07 	&-1.375\\
$2^{-8}$ &4.2870E+06	&-0.534         	&5.5401E+07  	&-1.473\\
$2^{-9}$ &1.2655E+07	&-1.561		&1.8020E+08 	&-1.701\\
 \hline
\end{tabular}
\end{table}
\end{center}
\begin{center}
\begin{figure}
\begin{tikzpicture}
\begin{loglogaxis}[
	xlabel={$T$},
	ylabel={},
	legend pos=outer north east
]
\addplot coordinates {
	 (1/512,1.8020E+08  )   (1/256,  5.5401E+07 )
	(1/128, 1.9956E+07)  (1/64,  7.6917E+06)  (1/32,  2.8531E+06)
	(1/16, 8.9903E+05 ) 
};

\addplot coordinates{
	 (1/512,1.2655E+07 )   (1/256,  4.2870E+06)
	(1/128, 2.9605E+06)  (1/64,  2.6687E+06)  (1/32,  1.5122E+06)
	(1/16, 3.5527E+05) 
};
\addplot [
    domain=1/512:1/16, 
    samples=100, 
    color=green,
]
{x^-3/2};
\legend{$\|v_{N, h}\|^2_{\bbR^{N}} + \|w_{N, h}\|^2_{\bbR^{N}}$,$\|u_h^*\|_{L^2(L^2(\Omega);0,T)}$, $y = x^{-3/2}$}
\end{loglogaxis}
\end{tikzpicture}
\caption{Logarithmic plots of $\|v_{N, h}\|^2_{\bbR^{N}} + \|w_{N, h}\|^2_{\bbR^{N}}$ vs. $\|u_h^*\|_{L^2(L^2(\Omega);0,T)}$ vs.$y = x^{-3/2}$ using Alg. 2.}
\end{figure}\label{FDMGraph}
\end{center}

\newpage
\section{\textbf{Conclusion}}

 The approximation of the null controller using both numerical schemes obey the same blow up rate of $\mathcal{O}(T^{-3/2})$. We also see that while the finite difference scheme (FD) gives better results approximating the solution at terminal time $T$, the finite element scheme (FE) is more stable computing the solution across different values of $T$.

\newpage
\section{\textbf{A Numerical Test Problem}}

In this section, we will derive an exact solution to problem \eqref{2ndODE} without the controller term $u$, that is

\begin{equation}\label{Homo}
\frac{d}{dt}\begin{bmatrix}
v \\
w
\end{bmatrix}=\mathcal{\overline{A}}\begin{bmatrix}
v \\
w
\end{bmatrix},\quad 
\begin{bmatrix}
v(0) \\
w(0)
\end{bmatrix} = \begin{bmatrix}
v_0 \\
w_0
\end{bmatrix}=\begin{bmatrix}
A\omega_0 \\
\omega_1
\end{bmatrix} \in L^2(\Omega) \times L^2(\Omega).
\end{equation}
where \begin{equation} \label{SDG1}
\mathcal{\overline{A}} = \begin{bmatrix}
0 & A  \\
-A & -\rho A 
\end{bmatrix}
\end{equation}
and the operator $A$ is the Laplacian defined in \eqref{DL}. The unique solution to \eqref{Homo} is given by 
\begin{equation} \label{ASOLREP}
\begin{bmatrix}
v(t) \\
w(t)
\end{bmatrix}=e^{\mathcal{\overline{A}}t}\begin{bmatrix}
v_0 \\
w_0
\end{bmatrix}
\end{equation}
In order to derive an explicit solution to \eqref{ASOLREP}, we need to compute the exponential matrix $e^{\mathcal{\overline{A}}t}$. Let $\{ \lambda_i, \phi_i  \}_{i = 1}^\infty$ be the eigenvalues and orthonormal eigenvectors for the operator $\mathcal{\overline{A}}$ defined in \eqref{SDG1}. Then 
\begin{equation*}
y(t) =\begin{bmatrix}
v(t) \\
w(t)
\end{bmatrix}\end{equation*}
must solve $y^{'}(t) = \mathcal{\overline{A}}y(t)$. Since $$y(t) =\begin{bmatrix}
\sum_i \alpha_i(t) \phi_i \\
\sum_i \beta_i(t) \phi_i
\end{bmatrix}$$ for some functions $\alpha_i,~\beta_i$ we then have 
\begin{equation}
\label{A1SOLREP}
\frac{d}{dt}\begin{bmatrix}
\sum_i \alpha_i(t) \phi_i \\
\sum_i \beta_i(t) \phi_i
\end{bmatrix} =\begin{bmatrix}
0 & A \\
-A & -\rho A
\end{bmatrix}\begin{bmatrix}
\sum_i \alpha_i(t) \phi_i \\
\sum_i \beta_i(t) \phi_i
\end{bmatrix}.
\end{equation}
By orthonormality, $\forall i = 1, 2, 3, ...$, 
\begin{equation}
\label{A2SOLREP}
\frac{d}{dt}\begin{bmatrix}
\sum_i \alpha_i(t) \phi_i \\
\sum_i \beta_i(t) \phi_i
\end{bmatrix} =M_i \begin{bmatrix}
\alpha_i(t) \phi_i \\
\beta_i(t) \phi_i
\end{bmatrix},
\end{equation}
where
\begin{equation}
\label{A3SOLREP}
M_i = \begin{bmatrix}
0 & \lambda_i \\
-\lambda_i & -\rho \lambda_i
\end{bmatrix} \qquad i=1, 2, 3, ...
\end{equation}
The eigenpairs for $M_i$ are 
\begin{equation}
\label{A4SOLREP}
\{ \eta_{i,1}, \begin{bmatrix}
-\frac{\rho}{2} + \frac{1}{2}\sqrt{\rho^2 - 4} \\
1
\end{bmatrix}  \} \cup \{ \eta_{i,2}, \begin{bmatrix}
-\frac{\rho}{2} - \frac{1}{2}\sqrt{\rho^2 - 4} \\
1
\end{bmatrix}  \}
\end{equation}
where 
\begin{align*}
\eta_{i,1} &= -\frac{\lambda_i}{2}\big( \rho + \sqrt{\rho^2 - 4}  \big),\\
\eta_{i,2} &= -\frac{\lambda_i}{2}\big( \rho - \sqrt{\rho^2 - 4}  \big).
\end{align*}
Denoting the similarity matrix 
\begin{equation}
\label{A5SOLREP}
S = \begin{bmatrix}
-\frac{\rho}{2} + \frac{1}{2}\sqrt{\rho^2 - 4} & -\frac{\rho}{2} - \frac{1}{2}\sqrt{\rho^2 - 4} \\
1 & 1
\end{bmatrix}
\end{equation}
using the change of variables $Sz = y,$ and the diagonalization argument gives us $$Sz^{'} = y^{'} = M_i Sz.$$ 
or $$z^{'} = S^{-1} M_i Sz = \Lambda z$$ where
\begin{equation*}
\Lambda = \begin{bmatrix}
\eta_{i, 1} & 0 \\
0 & \eta_{i, 2}
\end{bmatrix}, \qquad z = \begin{bmatrix}
c_{i,1}e^{\eta_{i, 1}t} \\
c_{i,2}e^{\eta_{i, 2}t}
\end{bmatrix}
\end{equation*}
Here $$\begin{bmatrix}
c_{i,1} \\
c_{i,2}
\end{bmatrix} = z(0) = S^{-1} y(0)$$ are constants. Observe that the constants $c_{i,1}, c_{i,2}$ can be found explicitly for $i = 1, 2, ...$ as
\begin{equation}
\label{A6SOLREP}
\begin{bmatrix}
c_{i,1}\\
c_{i,2}
\end{bmatrix} = S^{-1} \begin{bmatrix}
\alpha_i(0) \\
\beta_i(0)
\end{bmatrix} = \frac{1}{\sqrt{\rho^2 - 4}}\begin{bmatrix}
\alpha_i(0) + \frac{\beta_i(0)}{2}\big( \rho + \sqrt{\rho^2 - 4}  \big)\\
-\alpha_i(0) - \frac{\beta_i(0)}{2}\big( \rho - \sqrt{\rho^2 - 4}  \big)
\end{bmatrix}.
\end{equation}
Subsequently, we have an explicit formula for $\begin{bmatrix}
\alpha_i(t) \\
\beta_i(t)
\end{bmatrix}$ as
\begin{equation}
\label{A7SOLREP}
\begin{bmatrix}
\alpha_i(t) \\
\beta_i(t)
\end{bmatrix} = S \begin{bmatrix}
c_{i,1}e^{\eta_{i, 1}t} \\
c_{i,2}e^{\eta_{i, 2}t}
\end{bmatrix},
\end{equation}
From \eqref{A7SOLREP}, the solution $y(t)$ in \eqref{ASOLREP} can be written explicitly as
\begin{equation}
\label{A8SOLREP}
y(t) =\begin{bmatrix}
v(t) \\
w(t)
\end{bmatrix} = e^{\mathcal{\overline{A}}t}\begin{bmatrix}
v_0\\
w_0
\end{bmatrix} = \sum_{i=1}^{\infty} \begin{bmatrix}
\alpha_i(t)\phi_i\\
\beta_i(t)\phi_i
\end{bmatrix}.
\end{equation}
Now, let $\Omega = (0,\pi)^2$ and consider the problem

\begin{equation}
\label{A2ndODE}
\frac{d}{dt}\begin{bmatrix}
v(t) \\
w(t)
\end{bmatrix}=\mathcal{\overline{A}}\begin{bmatrix}
v(t) \\
w(t)
\end{bmatrix},\quad 
\begin{bmatrix}
v(.,0) \\
w(.,0)
\end{bmatrix} = \begin{bmatrix}
0 \\
\sin(2x)\sin(2y)
\end{bmatrix}.
\end{equation}
Recall that the Dirichlet Laplacian eigenpairs in $\Omega$ are 
\begin{equation*}
\{ \lambda_{mn} = m^2 + n^2,\; \phi_{mn} = \frac{2}{\pi}\sin(mx)\sin(ny)\}_{m,n=1}^{\infty}.
\end{equation*}
The initial data will be associated with 
\begin{equation*}
\lambda_{22} = 2^2 + 2^2 = 8,\; \phi_{22} = \frac{2}{\pi}\sin(2x)\sin(2y),
\end{equation*}
and subsequently we have for $i, j = 1, 2, ...$
\begin{equation*}
\alpha_{i, j}(0) = 0,\; \beta_{i, j}(0) = \left\{
\begin{array}{ll}
      \frac{\pi}{2}, & i = 0 = j  \\
      0, & otherwise\\
\end{array} 
\right.
\end{equation*}
Hence, the use of \eqref{A6SOLREP} and \eqref{A7SOLREP} would give us the functions $\begin{bmatrix}
\alpha_{i, j}(t) \\
\beta_{i, j}(t)
\end{bmatrix}.$ 

\noindent Now, we are in position to explicitly write the exact solution for the problem \eqref{A2ndODE} which will be used in our numerical experiments:
\begin{equation}
\label{A9SOLREP}
\begin{bmatrix}
v(t) \\
w(t)
\end{bmatrix} = \begin{bmatrix}
\Big(e^{-4t(\rho - \sqrt{\rho^2 - 4})} -    e^{-4t(\rho + \sqrt{\rho^2 - 4})}  \Big)\sin(2x)\sin(2y)\\
\Big(\Big( \frac{\sqrt{\rho^2 - 4}}{2} +  \frac{\rho}{2} \Big) e^{-4t(\rho + \sqrt{\rho^2 - 4})} +   \Big( \frac{\sqrt{\rho^2 - 4}}{2} -  \frac{\rho}{2} \Big)  e^{-4t(\rho - \sqrt{\rho^2 - 4})}\Big)\sin(2x)\sin(2y)
\end{bmatrix}.
\end{equation}
If we take $\rho = \frac{5}{2}$, then the expression \eqref{A9SOLREP} simplifies to 
\begin{equation}
\label{A10SOLREP}
\begin{bmatrix}
v(t) \\
w(t)
\end{bmatrix} = \begin{bmatrix}
\Big(e^{-4t} -    e^{-16t} \Big)\sin(2x)\sin(2y)\\
\Big(2e^{-16t} -    \frac{1}{2}e^{-4t} \Big)\sin(2x)\sin(2y)
\end{bmatrix}.
\end{equation}


\end{document}